\documentclass[hidelinks,onefignum,onetabnum]{siamart220329}
\usepackage{amsmath}
\usepackage{mathtools}
\usepackage{amssymb}
\usepackage{amsfonts}
\usepackage{amsxtra}
\usepackage{amstext}
\usepackage{amsbsy}
\usepackage{amscd}
\usepackage{graphicx}
\usepackage{float}
\usepackage{cite}
\usepackage{xcolor}
\usepackage{srcltx} %Inverse search
\usepackage{slashbox}
\usepackage{rotating}
\definecolor{darkblue}{rgb}{0.0,0.0,0.6}
\definecolor{darkgreen}{rgb}{0.0,0.6,0.0}

\usepackage{booktabs}            % to get fancier tables
\usepackage{url}                 % to typeset URLs
\usepackage{algorithmic}         % defines algorithmic environment

\usepackage{calc}

\usepackage{enumerate}

\usepackage{my_latex_commands}

\newcommand{\da}{\delta a}
\newcommand{\dz}{\delta z}
\newcommand{\dl}{\delta \ell}
\newcommand{\eps}{\varepsilon}
\newcommand{\dal}{(\delta a, \delta \ell)}
\newcommand{\io}{I_{0+}^{\gamma}}
\newcommand{\itt}{I_{T-}^{\gamma}}
\renewcommand{\ll}{\bar \ell}
\newcommand{\liv}{L^{\infty}(0,T;\mathbb{R}^{n})}
\newcommand{\li}{L^{\infty}(0,T;\mathbb{R}^{n\times n})}
\newcommand{\lia}{L^{\infty}(0,T;\mathbb{R}^{n\times n} \times \mathbb{R}^n)}
\newcommand{\cont}{C([0,T];\mathbb{R}^n)}
\newcommand{\conts}{C([0,t^*];\mathbb{R}^n)}
\newcommand{\contm}{C([0,T];\mathbb{R}^{n \times n})}
\newcommand{\ot}{(0,T;\mathbb{R}^n)}
\newcommand{\otm}{(0,T;\mathbb{R}^{n\times n})}
\newcommand{\ota}{(0,T;\mathbb{R}^{n\times n} \times \mathbb{R}^n)}
\newcommand{\dense}{\overset{d}{\embed}}
\newcommand{\alw}{(\widetilde a,\widetilde \ell)}
\newcommand{\al}{(a,\ell)}

\newcommand{\ale}{(a_\eps,\ell_\eps)}

\newcommand{\ds}{\,\d s}
\newcommand{\dt}{\,\d t}
\newsiamremark{remark}{Remark}
\newsiamremark{assumption}{Assumption}

\begin{document}

%\title{Strong stationarity for the control of a non-smooth fractional integral equation with control constraints: Application to a continuous DNN}

\title{Strong stationarity for optimal control problems with non-smooth integral equation constraints: Application to  continuous DNNs \thanks{HA is partially supported by NSF grant DMS-2110263 and the AirForce Office of Scientific Research under Award NO: FA9550-22-1-0248. LB is supported by the DFG grant BE 7178/3-1.
DW is partially supported by the DFG grant WA 3626/5-1.
}}
\date{\today}
\author{Harbir Antil\footnotemark[1], Livia\ Betz\footnotemark[2] and  Daniel Wachsmuth\footnotemark[2]}
\renewcommand{\thefootnote}{\fnsymbol{footnote}}
\footnotetext[1]{The Center for Mathematics and Artificial Intelligence, Department of Mathematical Sciences, George Mason University, Fairfax, VA 22030, USA}
\footnotetext[2]{Institut f\"ur Mathematik, Universit\"at W\"urzburg, 97074 W\"urzburg, Germany}
\renewcommand{\thefootnote}{\arabic{footnote}}
\maketitle

\begin{abstract}
Motivated by the residual type neural networks (ResNet), this paper studies optimal
control problems constrained by a non-smooth integral equation. Such non-smooth
equations, for instance, arise in the continuous representation of fractional
deep neural networks (DNNs). Here the underlying non-differentiable function
 is the ReLU or max function. The control enters in a nonlinear and
multiplicative manner and we additionally impose control constraints.
Because of the presence of the non-differentiable mapping, the application of
standard adjoint calculus is excluded.  We derive strong stationary conditions by
relying on the limited differentiability properties of the  non-smooth map. While traditional
approaches smoothen the non-differentiable function, no such smoothness is retained
in our final strong stationarity system. Thus, this work also closes a gap which currently
exists in continuous neural networks with ReLU type activation function.\end{abstract}


\section{Introduction}

In this paper, we establish strong stationary optimality conditions for the following control constrained  optimization problem
\begin{equation}\tag{P}\label{eq:optt}
 \left.
 \begin{aligned}
  \min_{(a,\ell) \in H^1(0,T;\mathbb{R}^{n \times n}) \times H^1(0,T;\mathbb{R}^n)} \quad & J(y,a,\ell)
  \\
 \text{s.t.} \quad &\partial^\gamma y(t)=f(a(t)y(t)+\ell(t)) \quad \text{a.e.\ in }(0,T),
  \\ \quad & y(0)=y_0,\
  \\ & \ell \in \KK,\end{aligned}
 \quad \right\}
\end{equation}where $f:\mathbb{R} \to \mathbb{R}$ is a \textit{non-smooth} non-linearity.
The objective functional is given by
\[J(y,a,\ell):=g(y(T)) +\frac{1}{2} \|a\|^2_{H^1(0,T;\mathbb{R}^{n \times n})} +\frac{1}{2} \|\ell\|^2_{H^1(0,T;\mathbb{R}^{n })},\]
where  $g: \mathbb{R}^n \to \mathbb{R}$ is a continuously differentiable function. The values $\gamma \in (0,1)$ and  $y_0 \in \mathbb{R}^n$ are fixed and the set $\KK \subset H^1\ot$ is convex and closed.  The symbol $\partial^\gamma$ denotes the fractional time derivative, more details are provided in the forthcoming sections. Notice that
the entire discussion in this paper also extends (and is new) for the case $\gamma = 1$, i.e., the standard time derivative. This has been substantiated with the help of several remarks throughout the paper.
%The case $\gamma=1$ can also be included and shall be discussed in the upcoming remarks.
Recently, optimal control of fractional ODEs/PDEs have received a significant interest, we refer to
the articles \cite{OPAgrawal_2004a,HAntil_EOtarola_AJSalgado_2015a} and the references therein.
The most generic framework is considered in \cite{unified}. However, none
of these articles deal with the non-smooth setting presented in this paper.

%\todo{I took out the part where you wanted to enumerate the main contributions and put instead: "main novelties"; note that the contributions are mentioned in detail at the end of the introduction. }
The essential feature of the problem under consideration is that the non-linearity $f$ is not necessarily G\^ateaux-differentiable, so that the standard methods for the derivation of qualified optimality conditions are not applicable here. In view of our goal to establish strong stationarity, the main novelties in this paper arise from:
\begin{itemize}
\item the presence of the fractional time derivative;
\item the fact that the controls  appear in the argument of the non-smooth non-linearity $f$;
\item the presence of control constraints (in this context, we are able to prove strong stationarity without resorting to  unverifiable ``constraint qualifications'').
\end{itemize}
All these challenges appear in  applications concerned with the control of neural networks.
The non-smooth and nonlinear function $f$ encompasses functions such as max or ReLU
arising in deep neural networks (DNNs). The objective function $J$ encompasses a generic class of
functionals such as cross entropy and least squares.
In fact, the optimal control problem \eqref{eq:optt} is motivated
by residual neural networks \cite{He_et_al_2016a,ruthotto2020deep} and fractional deep
neural networks \cite{HAntil_RKhatri_RLohner_DVerma_2020a,HAntil_HCElman_AOnwunta_DVerma_2021a,HAntil_HDiaz_Eherberg_2022a}. The control
constraints can capture the bias ordering notion recently introduced in \cite{fb}.
All existing approaches in the neural network setting assume differentiability of $f$ in deriving
the gradients via backpropagation. No such smoothness conditions are assumed in this
paper.

Deriving necessary optimality conditions is a challenging issue even in finite dimensions, where a special attention is given to MPCCs  (mathematical programs with complementarity constraints).
In \cite{ScheelScholtes2000} a detailed overview of various optimality conditions of different strength was introduced, see also \cite{HintermuellerKopacka2008:1} for the infinite-dimensional case. The  most rigorous stationarity concept is strong stationarity.
Roughly speaking, the strong stationarity conditions involve an optimality system, which is equivalent to the purely primal conditions saying that the directional derivative of the reduced objective in feasible directions is nonnegative (which is referred to as B-stationarity).

While there are plenty of contributions in the field of optimal control of
smooth problems, see e.g.\ \cite{troe} and the references therein, fewer papers are dealing with
non-smooth problems. Most of these works  resort to regularization or relaxation techniques to smooth the problem, see e.g.\ \cite{ Barbu:1981:NCD, ItoKunisch2008:1} and the references therein. The optimality systems derived in this way are of intermediate strength and are not expected to be of strong stationary type, since one always loses information when passing to the limit in the regularization scheme. Thus, proving strong stationarity for optimal control of  non-smooth problems requires direct approaches, which employ the limited differentiability properties of the control-to-state map. In this context, there are even less contributions. Based on the pioneering work  \cite{mp76} (strong stationarity  for  optimal control of elliptic VIs of obstacle type), most of them focus on elliptic VIs \cite{mp84, hmw13,   wachsm_2014, DelosReyes-Meyer, e_qvi, by18}; see also \cite{cc} (parabolic VIs of the first kind) and the more recent contribution \cite{brok_ch} (evolutionary VIs). Regarding strong stationarity for  optimal control of non-smooth PDEs, the literature is rather scarce and the only  papers known to the authors addressing this issue so far are  \cite{paper} (parabolic PDE), \cite{cr_meyer, quasi_nonsmooth, mcrf}  (elliptic PDEs) and  \cite{st_coup} (coupled PDE system). We point out that, in contrast to our problem, all the above mentioned works   feature controls which appears outside the non-smooth mapping. Moreover, none of these contributions  deals with   a fractional time derivative.

Let us give an overview of the structure and  the main results in this paper. After introducing  the notation, we present in section \ref{sec:p} some fractional calculus results which are needed throughout the paper. Then, section \ref{sec:c} focuses on the analysis of the state equation in \eqref{eq:optt}. Here we address the existence and uniqueness of so-called mild solutions, i.e., solutions of the associated integral Volterra equation. The properties of the respective control-to-state operator are investigated. In particular, we are concerned with the \textit{directional differentiability} of the solution mapping of the non-smooth integral equation  associated to the fractional differential equation in \eqref{eq:optt}.
While optimal control of nonlinear (and smooth) integral equations attracted much attention, see, e.g., \cite{Bittner1975,GoldbergTroltzsch1989,Wolfersdorf1976},
to the best of our knowledge, the sensitivity analysis of non-smooth integral equations has not been yet investigated in the literature.

In section \ref{sec:w} we show that the mild solution found in the previous section  is in fact strong. That is, the unique solution to the state equation in \eqref{eq:optt} is absolutely continuous, and it thus possesses a so-called \textit{Caputo-derivative}. We underline that, the only paper known to the authors which deals with optimal control and  proves the existence of strong solutions in the framework of fractional differential equations is \cite{unified}. In \cite{unified}, the absolute continuity of the mild solution of a fractional in time PDE (state equation) is shown by imposing pointwise (time-dependent) bounds on the time derivative of the control which then carry over to the time derivative of the state. We point out that we do not need such bounds in our case. Moreover, the result in this section stands by its own and is in fact not employed in the upcoming sections. However, it adds to the key novelties of the present paper.

Section \ref{sec:s} focuses on the main contribution of this paper, namely the strong stationarity  for the optimal control of \eqref{eq:optt}.
Via  a classical smoothening technique, we first  prove an auxiliary result (Lemma \ref{lem:tool}) which will serve as an essential tool in the context of establishing strong stationarity. Our  main Theorem \ref{thm:st} is then shown by extending  the  ``surjectivity'' trick from \cite{st_coup, paper}. In this context, we resort to a verifiable ``constraint qualification'' (CQ). We underline that this is satisfied by state systems describing neural networks  with the $\max$ or ReLu function. In addition, there are many other settings where the CQ can be a priori checked, as pointed out in Remark \ref{rem:cq} below. In a more general case, this CQ is the price to pay for imposing  constraints  on the control $\ell$ (and not on the control $a$), see Remark \ref{rem:surj}. As already emphasized in contributions where strong stationarity is investigated, CQs are  to be expected when examining control constrained problems \cite{mcrf,wachsm_2014} or, they may be required by the complex nature of the state system \cite{st_coup}.
At the end of section \ref{sec:s}
we gather some important remarks regarding the main result. A fundamental aspect resulting from the findings in this paper is that, when it comes to strong stationarity, the presence of \textit{more than one control} allows us to impose control constraints without having to resort to unverifiable CQs, see Remark \ref{rem:cc}.
Finally, we include in Appendix \ref{sec:a}  the proof of Lemma \ref{lem:tool}, for convenience of the reader.

%Equations of the type describe  the evolution of viscous processes, which is the application we are ultimately interested in this paper. In the end, we are able to derive strong stationary optimality conditions for the optimal control of a viscous EVI coupled with an elliptic PDE. ..hintergedanke...we keep in mind

\section*{Notation}
Throughout the paper, $T > 0$ is a fixed final time and $n \in \N$ is a fixed dimension. By $\|\cdot\|$ we denote the Frobenius norm. If $X$ and $Y$ are  linear normed spaces, $X \embed\embed Y$ means that  $X$ is compactly embedded in $Y$, {
while $X \dense Y$ means that $X$ is densely embedded in $Y$.} {The dual space of $X$ will be denoted by $X^*$. For the dual pairing between $X$ and $X^*$ we write $\dual{.}{.}_{X}.$} If $X$ is a Hilbert space,  $(\cdot,\cdot)_X$ stands for the associated scalar product. The {closed} ball in $X$ around $x \in X$ with radius $\alpha >0$ is denoted by $B_X(x,\alpha)$.
With a little abuse of notation, the Nemytskii-operators associated with the mappings considered in this paper will be denoted by the same symbol, even when considered with different domains and ranges. {We use sometimes the notation $h \lesssim g$ to denote $h \leq Cg$, for some constant $C>0$, when the dependence of the constant $C $ on some physical parameters is not relevant. }

\section{Preliminaries}\label{sec:p}
In this section we gather some fractional calculus tools  that are needed for our analysis.
\begin{definition}[Left and right Riemann-Liouville fractional integrals]\label{def:io}
For $\phi \in L^1(0,T; \mathbb{R}^n)$, we define
 \[I_{0+}^{\gamma} (\phi)(t):=\int_0^t \frac{(t-s)^{\gamma-1}}{\Gamma(\gamma)} \phi (s)\ds, \qquad I_{T-}^{\gamma} ({\phi})(t):=\int_t^T \frac{(s-t)^{\gamma-1}}{\Gamma(\gamma)} \phi (s)\ds \]for all $t\in [0,T]$.
Here $\Gamma$ is the Euler-Gamma function.
%Note that, if $\gamma=1,$ the above integrals coincide with the classical ones.
 \end{definition}
\begin{definition}[The Caputo fractional derivative]\label{def:caputo}
Let $y \in W^{1,1}(0,T;\mathbb{R}^n)$. The (strong) Caputo fractional derivative of order $\gamma \in (0,1)$ is given by
\[
\partial^\gamma y:=I_{0+}^{1-\gamma} y'.\]
\end{definition}

\begin{lemma}[Fractional integration by parts]\label{fip}{\cite[Lem.\,2.7a]{book}}
If  $\phi \in L^\varrho(0,T;\mathbb{R}^n)$ and $\psi \in L^\zeta(0,T;\mathbb{R}^n)$ with $\varrho,\zeta \in (1,\infty],$ $1/\varrho +1/\zeta \leq 1+\gamma$, then
\[\int_0^T {\psi^\top} I_{0+}^{\gamma} (\phi)\dt =\int_0^T {\phi^\top} I_{T-}^{\gamma}(\psi)\dt. \]
 \end{lemma}
 \begin{remark}\label{rem}Note that the identity in Lemma \ref{fip} implies that $I_{T-}^{\gamma}$ is the adjoint operator of $\io$.
 \end{remark}

 \begin{lemma}[Boundedness of fractional integrals]\label{lem:bound}\cite[Lem.\,2.1a]{book}
%\\(i)  If $\varrho<1/\gamma$, then $I_{0+}:L^\varrho(0,T;\mathbb{R}^n) \to L^\zeta(0,T;\mathbb{R}^n)$ is bounded for every $\zeta \in [1, \frac{\varrho}{1-\gamma \varrho}\Big)$.
%\\(ii) \ko{ If $\varrho>1/\gamma$, then $I_{0+}, I_{T-}:L^\varrho(0,T) \to C([0,T];\mathbb{R}^n)$ is bounded.}
%\\If $\gamma=1$, then $I_{0+}:L^\varrho(0,T;\mathbb{R}^n) \to W^{1,\varrho}(0,T;\mathbb{R}^n) $ is bounded.
The operators $I_{0+}^{\gamma}, I_{T-}^{\gamma}$ map $L^\infty(0,T;\mathbb{R}^n)$ to $C([0,T];\mathbb{R}^n)$. Moreover,  it holds
\begin{equation}\label{b}
\|\io\|_{\LL(L^\varrho(0,t;\mathbb{R}^n),L^\varrho(0,t;\mathbb{R}^n))}\leq  \frac{t^\gamma}{\gamma \Gamma (\gamma)} \quad \forall\,t\in [0,T]\end{equation}for all $\varrho \in [1,\infty]$. The same estimate is true for $\itt,$ cf.\ also Remark \ref{rem}.
 \end{lemma}
\begin{lemma}[Gronwall's inequality]\label{lem:g}\cite[Lem.\,6.3]{gronwall}
Let  $\varphi \in C([0,T];\mathbb{R}^n)$ with $\varphi \geq 0$. If
\begin{equation*}
\begin{aligned}
\varphi (t) \leq c_1 t^{\alpha-1}+c_2\int_0^t (t-s)^{\beta-1}\varphi(s)\ds
 \quad \forall\, t \in [0,T],
\end{aligned}
\end{equation*}where $c_1,c_2\geq 0$ are some constants and $ \alpha, \beta >0$,
then there is a positive constant $C=C(\alpha,\beta,T,c_2)$ such that
\begin{equation*}
\begin{aligned}
\varphi (t) \leq C c_1 t^{\alpha-1}
 \quad \forall\, t \in [0,T].
\end{aligned}
\end{equation*}
 \end{lemma}
{Finally, let us state a result that will be very useful throughout the entire paper.}
\begin{lemma}\label{lem:int}
Let $r \in [1,1/(1-\gamma)[$ be given for $\gamma \in (0,1)$.  Then
for each $t \in [0,T]$, we have
% $$[0,t] \ni s \mapsto (t-s)^{(\gamma-1)} \in L^{r}(0,t)$$ with
\[
\Big(\int_0^t (t-s)^{(\gamma-1)r} \ds\Big)^{1/r} { \leq \frac{t^{(\gamma-1)r+1}}{(\gamma-1)r+1} } \leq
\frac{T^{(\gamma-1)r+1}}{(\gamma-1)r+1} .
\]
\end{lemma}
\begin{proof}
By assumption, we have $(\gamma-1)r+1 > 0$ and
 \[\int_0^t (t-s)^{(\gamma-1)r} \ds=\frac{t^{(\gamma-1)r+1}}{(\gamma-1)r+1} \leq \frac{T^{(\gamma-1)r+1}}{(\gamma-1)r+1} \]
 follows from elementary calculations.
 The proof is complete.
\end{proof}

\section{Control-to-state operator}\label{sec:c}
%If $a, \ell \in L^\infty(0,T)$, then the state equation in (P) has a unique mild solution, i.e., there exists an unique $y \in C([0,T];\mathbb{R}^n)$ satisfying the non-smooth Volterra integral equation
%\begin{equation}\label{eq:1}
%y(t)=y_0+\int_0^t \frac{(t-s)^{\gamma-1}}{\Gamma(\gamma)} f(a(s)y(s)+\ell(s))\ds \quad \forall\, t \in [0,T].\end{equation} \ko{(proof..e.g.\,book of Diethelm, contraction argument)}.
%%This can be equivalently written as
%%$$y(t)=y_0+(g_{\gamma} \star f(ay+\ell))(t) \quad \forall\, t \in [0,T],$$
%%where $g_\gamma:=t^{\gamma-1}/{\Gamma(\gamma)}$ if $t \geq 0$ and $g_\gamma:=0$ if $t<0$.
%If $y \in W^{1,1},$ then \eqref{eq:1} is equivalent to the state equation.

In this section we address the properties of the solution operator of the state equation
\begin{equation}\label{eq:state}
\partial^\gamma y(t)=f(a(t)y(t)+\ell(t)) \quad \text{a.e.\ in }(0,T),
\qquad y(0)=y_0.
\end{equation}

Throughout the paper, {$\gamma \in (0,1)$, unless otherwise specified}. For all $z \in \mathbb{R}^n$, the non-linearity $f: \mathbb{R}^n \to \mathbb{R}^n$
satisfies
\[
f(z)_{i}=\widetilde f(z_{i}), \quad i=1,...,n,
\]
where $\widetilde f:\mathbb{R} \to \mathbb{R}$ is a non-smooth nonlinear function. For convenience we will denote both non-smooth functions by $f$; from the context it will always be clear which one is meant.

\begin{assumption}\label{assu:stand}
For the non-smooth mapping appearing in \eqref{eq:optt} we require:
 \begin{enumerate}
   \item\label{it:1}
The non-linearity $f: \mathbb{R}^n \to \mathbb{R}^n$
is globally Lipschitz continuous with  constant $L>0$, i.e.,
  \begin{equation*}
   \|f(z_1) - f(z_2)\| \leq L \|z_1 - z_2\| \quad \forall\, z_1, z_2 \in {\mathbb{R}^n} .
  \end{equation*}
%   \item\label{it:1}
%The non-linearity $f: \mathbb{R}^n \to \mathbb{R}^n$
%is Lipschitz continuous on bounded sets, i.e., for all $M>0$, there exists
%  a constant $L_M>0$ such that
%  \begin{equation*}
%   \|f(z_1) - f(z_2)\| \leq L_M \|z_1 - z_2\| \quad \forall\, z_1, z_2 \in B_{\mathbb{R}^n}(0,M) .
%  \end{equation*}
  \item\label{it:2} The function $f$ is directionally differentiable at every point, i.e.,
  \begin{equation*}
\lim_{\tau \searrow 0}
\Big\|\frac{f(z + \tau \,\dz) - f(z)}{\tau} - f'(z;\dz)\Big\| = 0 \quad \forall \, z,\dz \in \mathbb{R}^n.
  \end{equation*}
  \end{enumerate}
\end{assumption}
As a consequence of Assumption \ref{assu:stand} we have
 \begin{equation}\label{f_dir}
   \|f'(z;\delta z_1) - f'(z;\delta z_2)\| \leq L \|\delta z_1 - \delta z_2\| \quad \forall\, z, \delta z_1, \delta z_2 \in {\mathbb{R}^n} .
  \end{equation}

\begin{definition}Let 
Let $(a, \ell) \in L^\infty(0,T;\mathbb{R}^{n\times n} \times \mathbb{R}^n)$ be given. We say that $y \in C([0,T];\mathbb{R}^n)$ is a mild solution  of
the state equation \eqref{eq:state}
if it satisfies the following integral equation
 \begin{equation}\label{eq:1}
y(t)=y_0+\int_0^t \frac{(t-s)^{\gamma-1}}{\Gamma(\gamma)} f(a(s)y(s)+\ell(s))\ds \quad \forall\, t \in [0,T].\end{equation}
\end{definition}

\begin{remark}\label{rem:1}
One sees immediately that if $\gamma=1$ then $y \in W^{1,\infty}\ot$ and the mild solution is in fact a strong solution of the following ODE
\begin{equation*}
y'(t)=f(a(t)y(t)+\ell(t)) \quad \text{a.e.\ in }(0,T),
\qquad y(0)=y_0.
\end{equation*}
%In Section \ref{sec:w} below we will see that, if $\gamma >1/2,$ we can also talk about solutions satisfying \eqref{eq:state} in a strong sense, i.e., solutions
%which possess a Caputo fractional derivative.
\end{remark}

\begin{proposition}\label{prop:existence}
 For every $(a, \ell) \in L^\infty(0,T;\mathbb{R}^{n\times n} \times \mathbb{R}^n)$ there exists a unique mild solution $y \in C([0,T];\mathbb{R}^n)$ to the state equation
\eqref{eq:state}. \end{proposition}

\begin{proof}To show the existence of a mild solution in the general case $\gamma \in (0,1)$ we define the operator
\[F: C([0,t^*];\mathbb{R}^n)\ni z \mapsto \io(f(az+\ell)) \in \conts,\]
where $t^*$ will be computed such that $F$ is a contraction. Indeed, according to Lemma \ref{lem:bound}, $F$ is well-defined (since $f$ maps bounded sets to bounded sets). Moreover, by applying \eqref{b} with $\varrho=\infty$, and using Assumption~\ref{assu:stand}, we see that
\begin{align*}
\|F(z_1)-F(z_2)\|_{\conts} &\leq \frac{(t^*)^\gamma}{\gamma \Gamma (\gamma)}\|f(az_1+\ell)-f(az_2+\ell)\|_{\conts}
\\&\leq \frac{(t^*)^\gamma}{\gamma \Gamma (\gamma)}L\, \|a\|_{L^\infty(0,T;\mathbb{R}^{n\times n})}\|z_1-z_2\|_{\conts}
\end{align*}for all $z_1,z_2 \in \conts.$
%note that here we also relied on Assumption \ref{assu:stand}.\ref{it:1} with $M(t^*):=\|a\|_{L^\infty(0,T;\mathbb{R}^{n\times n})}\max\{\|z_1\|_{\conts},\|z_2\|_{\conts}\}+\|\ell\|_{L^\infty(0,T;\mathbb{R}^{n})}$ \ko{M depends on t*}.
Thus, $F:\conts \to \conts$ is a contraction  provided that $\frac{(t^*)^\gamma}{\gamma \Gamma (\gamma)}  L\,\|a\|_{L^\infty(0,T;\mathbb{R}^{n\times n})} <1$. If $\frac{T^\gamma}{\gamma \Gamma (\gamma)}  L\, \|a\|_{L^\infty(0,T;\mathbb{R}^{n\times n})} <1$, the proof is complete. Otherwise we fix $t^*$ as above and conclude that $z=F(z)$ admits  a unique solution in $\conts$, which for later purposes, is denoted by $\widetilde y$. To prove that this solution can be extended on the whole given interval $[0,T],$ we use a concatenation argument. We define
\begin{align*}
\widehat F: C([t^*,2t^*];\mathbb{R}^n)\ni z \mapsto  y_0+\int_{t^*}^\cdot \frac{(\cdot-s)^{\gamma-1}}{\Gamma(\gamma)} f(a(s)z(s)+\ell(s))\ds
\\+\int_{0}^{t^*} \frac{(\cdot-s)^{\gamma-1}}{\Gamma(\gamma)} f(a(s)\widetilde y(s)+\ell(s))\ds \in C([t^*,2t^*];\mathbb{R}^n).
\end{align*}
Using a simple coordinate transform, we can apply
again \eqref{b} with $\varrho=\infty$ on the interval $(t^*,2t^*)$, and we have
\begin{multline*}
 \|\widehat F(z_1)-\widehat F(z_2)\|_{C([t^*,2t^*];\mathbb{R}^n)} \\
\begin{aligned}
& \le \sup_{t\in [t^*,2t^*]}
\int_{t^{*}}^{t} \frac{(t-s)^{\gamma-1}}{\Gamma(\gamma)}L\, \|a\|_{L^\infty(0,T;\mathbb{R}^{n\times n})}\|(z_1-z_2)(s)\|_{\mathbb{R}^n}\ds
\\&\leq \frac{(t^*)^\gamma}{\gamma \Gamma (\gamma)} L\, \|a\|_{L^\infty(0,T;\mathbb{R}^{n\times n})}\|z_1-z_2\|_{C([t^*,2t^*];\mathbb{R}^n)}
\end{aligned}
\end{multline*}
for all $z_1,z_2 \in C([t^*,2t^*];\mathbb{R}^n).$ Since $t^{*}$ was fixed so that $\frac{(t^*)^\gamma}{\gamma \Gamma (\gamma)}  L\,\|a\|_{L^\infty(0,T;\mathbb{R}^{n\times n})} <1$, we deduce that $z=\widehat F(z)$ admits  a unique solution $\widehat y$ in $C([t^*,2t^*];\mathbb{R}^n)$. By concatenating the local solutions found on the intervals $[0,t^*]$ and $[t^*,2t^*]$ one obtains a unique continuous function on $[0,2t^*]$ which satisfies the integral equation \eqref{eq:1}. Proceeding further  in the exact same way, one finds that \eqref{eq:1} has a unique solution in $\cont$.
\end{proof}
%
%\begin{remark}The arguments and the outcome of the proof stay the same, if  $a,\ell \in L^\varrho(0,T;\mathbb{R}^n)$, with $\varrho >1/\gamma$, in the case $\gamma \in (0,1)$ and $\varrho=1$, in the case  $\gamma=1$.
%%This is due to the result in Thm 1 which allows us to conclude existence of continuous solutions for slightly  less regular right-hand sides.
%However, since $a,\ell \in H^1(0,T;\mathbb{R}^n) \subset L^\infty(0,T)$, we choose to restrict our proof to the case where the right-hand sides are essentially bounded.
%\end{remark}
\begin{proposition}[$S$ is Locally Lipschitz]\label{prop:lip}
The solution operator associated to \eqref{eq:state} \[S: \lia  \ni (a, \ell) \mapsto y \in C([0,T];\mathbb{R}^n)\] is locally Lipschitz continuous in the following sense:
 For every $M>0$ there exists a constant $L_M > 0$ such that
 \begin{equation}\label{eq:loclipS}
  \|S(a_1, \ell_1) - S(a_2, \ell_2)\|_{C([0,T];\mathbb{R}^n)} \leq L_M (\| a_1-a_2\|_{L^{\infty}\otm}+\| \ell_1-\ell_2\|_{L^{\infty}\ot}),
 \end{equation}
 for all $(a_1, \ell_1), (a_2, \ell_2) \in {B_{\lia}(0,M)}$.
\end{proposition}\begin{proof}
First we show that $S$ maps bounded sets to bounded sets. To this end, let $M>0$ and $(a, \ell) \in \lia$ be arbitrary but fixed such that  $\|(a,\ell)\|_{\lia}\leq M$. From \eqref{eq:1} we have
\begin{equation*}
\begin{aligned}
\| y(t)\| &\leq  \|y_0\|+\int_0^t \frac{(t-s)^{\gamma-1}}{\Gamma(\gamma)}
( \|f(0)\| + L\|a\|_{\li}\|y(s)\|+ \|\ell\|_{L^{\infty}(0,T;\mathbb{R}^n)} )\ds
\\& \le c_1+\int_0^t \frac{(t-s)^{\gamma-1}}{\Gamma(\gamma)} L\,M \|y(s)\|\ds
 \quad \forall\, t \in [0,T],
\end{aligned}
\end{equation*}
with
\[
c_1 = \|y_0\| + \frac{T^\gamma}\gamma (\|f(0)\| +\|\ell\|_{L^{\infty}(0,T;\mathbb{R}^n)} ),
\]
where  we used Assumption \ref{assu:stand}.\ref{it:1} and Lemma \ref{lem:int} with $r=1$. By means of Lemma \ref{lem:g} we deduce
\begin{equation}\label{eq:bound_y}
\| y\|_{C([0,T];\mathbb{R}^n)} \leq c(\|y_0\|, \|f(0)\|, L,M,\gamma,T)=:c_M.
\end{equation}
Now, let $M>0$  be further arbitrary but fixed. Define $y_k:=S(a_k,\ell_k),\ k=1,2$ and consider $\|(a_k,\ell_k)\|_{\lia}\leq M, \ k=1,2.$
Subtracting the  integral formulations associated to each $k$ and using the Lipschitz continuity of $f$ with constant $L$ yields
for $t \in [0,T]$
\begin{equation*}
\begin{aligned}
\| (y_1-y_2)(t)\| &\leq \int_0^t \frac{(t-s)^{\gamma-1}}{\Gamma(\gamma)} L\,\|a_1(s)y_1(s)-a_2(s) y_2(s)+\ell_1(s)-\ell_2(s)\|\ds
\\ \leq&  \int_0^t \frac{(t-s)^{\gamma-1}}{\Gamma(\gamma)}L\, \|a_1\|_{\li}\|y_1(s)- y_2(s)\|\ds
\\&\,  +
\int_0^t \frac{(t-s)^{\gamma-1}}{\Gamma(\gamma)} L\|y_2\|_{\liv}( \|a_1(s)-a_2(s)\|+\|\ell_1(s)-\ell_2(s)\|)\ds
\\
\leq &  \int_0^t \frac{(t-s)^{\gamma-1}}{\Gamma(\gamma)}L\,M\, \|y_1(s)- y_2(s)\|\ds \\
&
+ \frac{t^\gamma}\gamma  L\,(c_M \|a_1-a_2\|_{\li} + \|\ell_1-\ell_2\|_{\liv}),
\end{aligned}
\end{equation*}
where in the last inequality we used \eqref{eq:bound_y} and Lemma \ref{lem:int} with $r=1$. Now, Lemma \ref{lem:g} implies that
\begin{equation*}
\|y_1- y_2\|_{C([0,T];\mathbb{R}^n)} \lesssim c_M(\| a_1-a_2\|_{\li}+\| \ell_1-\ell_2\|_{\liv}),
\end{equation*}
which completes the proof.
\end{proof}
\begin{theorem}[$S$ is directionally differentiable]
\label{prop:dir}
The control to state operator \[S: \lia  \ni (a, \ell) \mapsto y \in C([0,T];\mathbb{R}^n)\] is directionally differentiable  with directional derivative given by the unique solution $\delta y \in C([0,T];\mathbb{R}^n)$ of the following integral equation
\begin{equation}\label{eq:2}
\delta y(t)=\int_0^t \frac{(t-s)^{\gamma-1}}{\Gamma(\gamma)} f'(a(s)y(s)+\ell(s);a(s)\delta y(s)+\delta a(s) y(s)+\delta \ell(s))\ds \quad \forall\, t \in [0,T],\end{equation} i.e., $\delta y=S'((a,\ell);(\delta a, \delta \ell))$ for all $(a,\ell), \dal \in \lia$.
\end{theorem}
\begin{proof}We first  show that \eqref{eq:2} is uniquely solvable. To this end, we argue as in the proof of Proposition \ref{prop:existence}. From Lemma \ref{lem:bound} we know that the operator
\[F: C([0,t^*];\mathbb{R}^n)\ni z \mapsto \io( f '(ay+\ell;az+\delta a y+\delta \ell)) \in \conts,\]
 is well defined since $f '(ay+\ell;az+\delta a y+\delta \ell) \in \liv$ for $z \in \liv$ (see \eqref{f_dir}). By employing the Lipschitz continuity of $ f '(ay+\ell;\cdot)$ with constant $L$, one obtains the exact same estimate as in the proof of Proposition \ref{prop:existence} and the remaining arguments stay the same.

Next we focus on proving that $\delta y$ is the directional derivative of $S$ at $(a,\ell)$ in direction $\dal$. For $\tau \in (0,1]$ we define $y^\tau:=S(a+\tau \delta a,\ell+\tau \delta \ell), \ (a^\tau,\ell^\tau):=(a+\tau \delta a,\ell+\tau \delta \ell)$.
%First we show that $S$ maps bounded sets to bounded sets. To this end, let $M>0$ and $a, \ell \in L^\infty(0,T)$ be arbitrary but fixed with $\|(\ell,a)\|_{L^{\infty}(0,T)}\leq M$.
From \eqref{eq:1} we have
\begin{multline*}
\Big(\frac{y^\tau- y}{\tau}-\delta y\Big)(t) \\
\begin{aligned}
&= \frac{1}{\tau}\,\int_0^t \frac{(t-s)^{\gamma-1}}{\Gamma(\gamma)} (f(a^\tau(s)y^\tau(s)+\ell^\tau(s))-f(a(s)y(s)+\ell(s)))\ds
\\&\quad -\int_0^t \frac{(t-s)^{\gamma-1}}{\Gamma(\gamma)} f '(\underbrace{a(s)y(s)+\ell(s)}_{=:h(s)};\underbrace{a(s)\delta y(s)+\delta a(s) y(s)+\delta \ell(s)}_{=:\delta h(s)})\ds
\\&=\int_0^t \frac{(t-s)^{\gamma-1}}{\Gamma(\gamma)}  \frac{1}{\tau}\,[f(a^\tau(s)y^\tau(s)+\ell^\tau(s))-f((h+\tau \delta h)(s))]\ds
\\&\quad +\int_0^t \frac{(t-s)^{\gamma-1}}{\Gamma(\gamma)} \Big( \underbrace{\frac{1}{\tau} [f((h+\tau \delta h)(s))-f(h(s))]-f '(h(s);\delta h(s))}_{=:B_\tau(s)} \Big)\ds
\end{aligned}
\end{multline*}
for all $t \in [0,T].$
Since $f$ is Lipschitz continuous with constant $L$ we get
\begin{equation}\label{eq:deltay}
\begin{aligned}
\Big\|\Big(\frac{y^\tau- y}{\tau}-\delta y\Big)(t)\Big\| & \leq  \int_0^t \frac{(t-s)^{\gamma-1}}{\Gamma(\gamma)} L\,\Big(\frac{1}{\tau}\,\|a^\tau(s)y^\tau(s)+\ell^\tau(s)-h(s)-\tau \delta h(s)\|\Big)\ds
\\&\quad +\|B_\tau\|_{L^q(0,t;\mathbb{R}^n)}
%\\&\leq  \int_0^t \frac{(t-s)^{\gamma-1}}{\Gamma(\gamma)} (\frac{1}{\tau}\,((a+\tau \delta a)(s)y^\tau(s)-a(s)y(s)-\tau(a \delta y+\delta a y)(s)))\ds
%\\&\quad +\int_0^t \frac{(t-s)^{\gamma-1}}{\Gamma(\gamma)} B_\tau(s)\ds
 \quad \forall\, t \in [0,T],
\end{aligned}
\end{equation}
where $q=r' <\infty$, with $r$  given by Lemma \ref{lem:int}. Note that, in view of the directional differentiability of $f$ combined with Lebesgue dominated convergence theorem it holds
\begin{equation}\label{btau}
B_\tau \to 0 \quad \text{ in }L^{q}(0,T) \quad \text{ as }\tau \searrow 0.\end{equation}
Now, let us take a closer look at the term
\begin{equation*}
\begin{aligned}&
\frac{1}{\tau}\,\|a^\tau(s)y^\tau(s)+\ell^\tau(s)-h(s)-\tau \delta h(s)\|
\\&= \frac{1}{\tau}\,\big\|(a+\tau \delta a)(s)y^\tau(s)-a(s)y(s)-\tau(a \delta y+\delta a y)(s)\big\|
\\&=\Big\| a(s)\Big(\frac{y^\tau- y}{\tau}-\delta y\Big)(s)+\delta a(s) (y^\tau(s)-y(s))\Big\|
\\&\leq \|a\|_{\li}\Big\|\Big(\frac{y^\tau- y}{\tau}-\delta y \Big)(s)\Big\|
\\&\qquad \quad +\underbrace{\tau L_M\, \|\delta a\|_{\li} (\|\delta a\|_{\li} +\|\delta \ell\|_{\liv})}_{=:b_{\tau}} \quad \forall\,s\in [0,T],\end{aligned}
\end{equation*}
where in the last  inequality we used the Lipschitz continuity of $S$, cf.\, Proposition \ref{prop:lip}, with $M:=\|a\|_{\li} +\|\delta a\|_{\li}+\| \ell\|_{\liv} +\|\delta \ell\|_{\liv}$.  Going back to \eqref{eq:deltay}, we see that
\begin{multline*}
\Big\|\Big(\frac{y^\tau- y}{\tau}-\delta y\Big)(t)\Big\|  \leq \|B_\tau\|_{L^q(0,T;\mathbb{R}^n)}+L\,b_\tau
\\ +L\,\|a\|_{L^{\infty}\otm}\int_0^t \frac{(t-s)^{\gamma-1}}{\Gamma(\gamma)} \Big\|\Big(\frac{y^\tau- y}{\tau}-\delta y\Big)(s)\Big\| \ds
%\\&\leq  \int_0^t \frac{(t-s)^{\gamma-1}}{\Gamma(\gamma)} (\frac{1}{\tau}\,((a+\tau \delta a)(s)y^\tau(s)-a(s)y(s)-\tau(a \delta y+\delta a y)(s)))\ds
%\\&\quad +\int_0^t \frac{(t-s)^{\gamma-1}}{\Gamma(\gamma)} B_\tau(s)\ds
 \quad \forall\, t \in [0,T],
\end{multline*}
where we relied again on Lemma \ref{lem:int} with $r=1.$
In light of \eqref{btau}, Lemma \ref{lem:g} finally implies
\begin{equation*}
\begin{aligned}
\Big\|\frac{y^\tau- y}{\tau}-\delta y\Big\|_{C([0,T];\mathbb{R}^n)}& \lesssim \|B_\tau\|_{L^q(0,T;\mathbb{R}^n)}+L\,b_\tau    \to 0 \qquad \text{as } \tau \searrow 0.
\end{aligned}
\end{equation*}
The proof is now complete.
\end{proof}

\begin{remark}
Note that in the case $\gamma=1$, one obtains by arguing exactly as in the proof of Theorem \ref{prop:dir} that
 \[S: \lia  \ni (a, \ell) \mapsto y \in W^{1,\infty}(0,T;\mathbb{R}^n)\] is directionally differentiable  with directional derivative given by the unique solution $\delta y \in W^{1,\infty}(0,T;\mathbb{R}^n)$ of the following ODE
\begin{equation*}\label{eq:2_1}
\delta y'(t)=f '(a(t)y(t)+\ell(t);a(t)\delta y(t)+\delta a(t) y(t)+\delta \ell(t)) \quad \text{f.a.a.\,} t \in (0,T),\quad \delta y(0)=0.\end{equation*}\end{remark}

 \section{Strong solutions}\label{sec:w}
Next we prove that the state equation
\begin{equation}\label{eq:caputo}
\partial^\gamma y(t)=f(a(t)y(t)+\ell(t)) \quad \text{a.e.\ in }(0,T),
\qquad y(0)=y_0 ,
\end{equation}
admits in fact strong solutions, i.e., solutions that possess a so-called Caputo derivative, see Definition \ref{def:caputo}.{
\begin{definition}\label{strong_sol}
We say that $y \in W^{1,1}(0,T;\mathbb{R}^n)$ is a strong solution to \eqref{eq:caputo} if
\[
I_{0+}^{1-\gamma} y'=f(ay+\ell) \quad \text{a.e.\ in }(0,T),
\qquad y(0)=y_0.\]\end{definition}}
The following well known result is a consequence of the identity $\io I_{0+}^{1-\gamma}y' =I_{0+}^{1}y'=y$, which is implied by the semigroup property of the fractional integrals cf.\,e.g. \cite[Lem.\,2.3]{book} and  Definition \ref{def:io}.
\begin{lemma}
A function $y \in W^{1,1}(0,T;\mathbb{R}^n)$ is a strong solution of \eqref{eq:caputo} if and only if it satisfies the integral formulation \eqref{eq:1}.
\end{lemma}

\begin{theorem}\label{prop:w11}
Let {$\gamma \in (0,1)$ and $r\in [1,\frac{1}{1-\gamma})$} beg given. For each $(a,\ell) \in W^{1,\varrho} \otm \times W^{1,\varrho} \ot$, $\varrho>1$,  \eqref{eq:caputo} admits a unique strong solution $y \in W^{1,\zeta}(0,T;\mathbb{R}^n)$, where $\zeta=\min\{r,\varrho\}>1$.
%If, in addition, $f(a(0)y_0+\ell(0))=0$ is satisfied, then $y \in C^{1}([0,T];\mathbb{R}^n)$.
 \end{theorem}
 \begin{proof}
  Let $t \in [0,T]$ and $h\in (0,1]$ be arbitrary but fixed. Note that the existence of a unique solution $y \in C([0,T+1];\mathbb{R}^n)$  is guaranteed by Proposition \ref{prop:existence}; this solution coincides with the mild solution of \eqref{eq:state} on the interval $[0,T]$.  From \eqref{eq:1} we have
 \[   y(t+h)-y(t)=\int_0^{t+h} \frac{s^{\gamma-1}}{\Gamma(\gamma)} f(ay+\ell)(t+h-s)\ds
-\int_0^{t} \frac{s^{\gamma-1}}{\Gamma(\gamma)} f(ay+\ell)(t-s)\ds,
 \]
 which implies
 \begin{equation*}
\begin{aligned}
\|y(t+h)-y(t)\| & \leq \int_t^{t+h} \frac{s^{\gamma-1}}{\Gamma(\gamma)} f(ay+\ell)(t+h-s)\ds
\\&+\int_0^{t} \frac{s^{\gamma-1}}{\Gamma(\gamma)}L\, \|(ay+\ell)(t+h-s)-(ay+\ell)(t-s)\|\ds
\\&:=z_1(t,h)+z_2(t,h).
\end{aligned}
\end{equation*}
For $z_1$ and $z_2$ we find the following estimates
\begin{equation*}
\begin{aligned}
\frac{z_1(t,h)}{h} \lesssim \frac{1}{h}\int_t^{t+h} s^{\gamma-1}\ds \leq t^{\gamma-1},
\end{aligned}
\end{equation*}\begin{equation*}
\begin{aligned}
z_2(t,h) &\lesssim \int_0^{t} \frac{s^{\gamma-1}}{\Gamma(\gamma)} \|(a(t+h-s)-a(t-s))y(t+h-s)\|\ds
\\&\qquad +\int_0^t \frac{s^{\gamma-1}}{\Gamma(\gamma)} \|a(t-s)(y(t+h-s)-y(t-s))\|\ds
\\&\qquad +\int_0^{t} \frac{s^{\gamma-1}}{\Gamma(\gamma)}\|\ell(t+h-s)-\ell(t-s)\|\ds
\\&\lesssim \int_0^{t} \frac{s^{\gamma-1}}{\Gamma(\gamma)} \|(a(t+h-s)-a(t-s))\|\ds \\&\qquad +
\int_0^{t} \frac{s^{\gamma-1}}{\Gamma(\gamma)} \|y(t+h-s)-y(t-s)\|\ds
\\&\qquad +\int_0^{t} \frac{s^{\gamma-1}}{\Gamma(\gamma)}\|\ell(t+h-s)-\ell(t-s)\|\ds,
\end{aligned}
\end{equation*}where we relied on the fact that $y,a,$ and $\ell$ are essentially bounded.
Altogether we have
\begin{equation}\label{eq:w2}
\begin{aligned}
\Big\|\frac{y(t+h)-y(t)}{h}\Big\| & \lesssim t^{\gamma-1}+\int_0^t (t-s)^{\gamma-1} \Big\|\frac{\ell(s+h)-\ell(s)}{h}\Big\| \ds
\\&\quad +\int_0^t (t-s)^{\gamma-1}\Big\|\frac{a(s+h)-a(s)}{h}\Big\|\ds
\\ \quad&+\int_0^t (t-s)^{\gamma-1} \Big\|\frac{y(s+h)-y(s)}{h}\Big\|\ds
% \\& =: B_h(t)
% \\ \quad&+\int_0^t (t-s)^{\gamma-1} \Big\|\frac{y(s+h)-y(s)}{h}\Big\|\ds
\end{aligned}
\end{equation}
for all $t \in [0,T]$ and all $h \in (0,1]$.
Let us define
\[
 B_h(t):=t^{\gamma-1}+\int_0^t (t-s)^{\gamma-1} \Big\|\frac{\ell(s+h)-\ell(s)}{h}\Big\| \ds +\int_0^t (t-s)^{\gamma-1}\Big\|\frac{a(s+h)-a(s)}{h}\Big\|\ds.
\]
Since $(a,\ell) \in W^{1,\varrho} \otm \times W^{1,\varrho} \ot$, and $\zeta=\min\{r,\varrho\}$, where $r$ is given by Lemma \ref{lem:int}, we can  estimate $B_h$ as follows
\begin{equation}\label{b_h}
\begin{aligned}
\|B_{h}\|_{L^\zeta(0,T;\mathbb{R})}
& \le \frac{T^{(\gamma-1)r+1}}{(\gamma-1)r+1}+ \frac{T^\gamma}{\gamma \Gamma (\gamma)} \Big\|\frac{\ell(\cdot+h)-\ell(\cdot)}{h}\Big\|_{L^\zeta(0,T;{\mathbb{R}^n})}
\\
&\qquad
+\frac{T^\gamma}{\gamma \Gamma (\gamma)}\Big\|\frac{a(\cdot+h)-a(\cdot)}{h}\Big\|_{L^\zeta(0,T; {\mathbb{R}^{n\times n}})}
\\
&\le  \frac{T^{(\gamma-1)r+1}}{(\gamma-1)r+1}+ \frac{T^{\gamma+\zeta^{-1}-\varrho^{-1}}}{\gamma \Gamma (\gamma)} (\|\ell'\|_{L^\varrho(0,T; {\mathbb{R}^n})}+\|a'\|_{L^\varrho(0,T; {\mathbb{R}^{n\times n}})}),
\end{aligned}
\end{equation}
where we relied on Lemmas \ref{lem:int}, \ref{lem:bound} and  \cite[Thm.\,3, p.\,277]{evans}.
Hence, $\{B_h\}$ is uniformly bounded in $L^\zeta(0,T;\mathbb{R})$ with respect to $h\in (0,1]$.
Further, the generalized Gronwall inequality of \cite[Lemma 7.1.1]{henry},
see also \cite[Corollary 1]{YeGaoDing2007}, applied
to \eqref{eq:w2} yields
\begin{equation*}\label{gronw}
\Big\|\frac{y(t+h)-y(t)}{h}\Big\| \lesssim B_h(t)
+\int_0^t \Big[\sum_{n=0}^\infty \frac{\Gamma(\gamma)^n}{\Gamma(n\gamma)}(t-s)^{n\gamma-1}B_h(s) \Big] \ds
\end{equation*}
for all $t \in [0,T]$ and all $h \in (0,1]$.
Using monotone convergence theorem, we can exchange the order of integration and summation to get
\[
\Big\|\frac{y(t+h)-y(t)}{h}\Big\| \lesssim  B_h(t)
+\sum_{n=0}^\infty \Gamma(\gamma)^n (I_{0+}^{n\gamma}B_h)(t),
\]
where we used the definition of $I_{0+}^{n\gamma}$ from Definition \ref{def:io}.
Applying Lemma \ref{lem:bound} we obtain
\begin{equation*}
\begin{aligned}
\Big\|\frac{y(\cdot+h)-y(\cdot)}{h}\Big\|_{L^\zeta(0,T;{\mathbb{R}^{n})}}&\lesssim \|B_{h}\|_{L^\zeta(0,T;\mathbb{R})}+\sum_{n=0}^\infty {\Gamma(\gamma)^n}
\|I_{0+}^{n\gamma}B_h\|_{L^\zeta(0,T;\mathbb{R})}
\\&\leq \|B_{h}\|_{L^\zeta(0,T;\mathbb{R})} + \sum_{n=0}^\infty {\Gamma(\gamma)^n}
\frac{T^{n\gamma}}{n\gamma \Gamma(n\gamma)}\|B_h\|_{L^\zeta(0,T;\mathbb{R})}
\\&=\|B_{h}\|_{L^\zeta(0,T;\mathbb{R})}[1+E_{\gamma,1}(\Gamma(\gamma)T^\gamma)] \quad \forall\,h\in(0,1],
\end{aligned}
\end{equation*}where $E_{\gamma,1}(z)= \sum_{n=0}^\infty
\frac{z^{n}}{\Gamma(n\gamma+1)}<\infty$
is the celebrated Mittag-Leffler function; note that here we used $n\gamma \Gamma(n\gamma)= \Gamma(n\gamma+1)$.
Since $\{B_h\}$ is uniformly bounded in $L^\zeta(0,T;\mathbb{R})$, see \eqref{b_h},
we obtain that the difference quotients of $y$ are uniformly bounded in $L^\zeta(0,T;\mathbb{R})$ with respect to $h\in(0,1]$.
Hence, $y$
has a weak derivative in $L^\zeta(0,T;\mathbb{R})$ by \cite[Thm.\,3, p.\,277]{evans}. The proof is now complete.
%
%
%To show the second assertion, we see that by differentiating the equation  for the mild solution  \eqref{eq:1} one  has
% \begin{equation}\label{eq:1_deriv}
%y'(t)=t^{\gamma-1} f(a(0)y_0+\ell(0))+\int_0^t \frac{(t-s)^{\gamma-1}}{\Gamma(\gamma)} f'((ay+\ell ) (s);a'(s)y(s)+ay'(s)+\ell '(s))\ds \quad \text{for a.a.}\, t \in (0,T).\end{equation}
%On the other hand, since $\gamma>1/2$, by assumption, we have that the fix point equation
% \begin{equation}\label{eq:1_derivv}
%z(t)=\int_0^t \frac{(t-s)^{\gamma-1}}{\Gamma(\gamma)} f'((ay+\ell ) (s);a'(s)y(s)+az(s)+\ell '(s))\ds \quad \text{for a.a.}\,t \in (0,T)\end{equation}
%admits a unique solution in $\cont$, see \cite[Thm.\,1.3]{fip}.
%Thus, since If $f(a(0)y_0+\ell(0))=0$, by assumption, we have that $y$ belongs to
%\begin{equation}
%\begin{aligned}
%C^1[0,T], &\qquad \text{if }1/2<\gamma<1,
%\end{aligned}
%\end{equation}
%The proof is now complete.
\end{proof}

\begin{remark}
We remark that the degree of smoothness of  the right-hand sides $a,\ell$ does not necessarily  carry over to  the strong solution $y$ (unless a certain compatibility condition is satisfied, see Remark \ref{rem:bel} below). This is  in accordance with observations made in literature, see e.g., \cite[Ex.\,6.4, Rem\,6.13, Thm\,6.27]{diethelm} (fractional ODEs) and \cite[Cor.\,2.2]{stynes} (fractional in time PDEs). Indeed, for large values for $\varrho$ and small values of $\gamma$ tending to $0$, the strong solution $y \in W^{1,\zeta}(0,T;\mathbb{R}^n)$, where $\zeta=r \in (1,1/(1-\gamma))$ is close to $1$. However, as $\gamma$ approaches the value $1$, one can expect  the strong solutions to become as regular as their right-hand sides.
This can be seen in the case $\gamma=1$, where the smoothness of the strong solution improves as the smoothness of $a,\ell$ does so. Note that in this particular situation the solution of \eqref{eq:caputo} is in fact far more regular than as in the statement in  Theorem \ref{prop:w11}, see Remark \ref{rem:1}.\end{remark}
\begin{remark}[Compatibility condition]\label{rem:bel}
If $f(a(0)y_0+\ell(0))=0$, then the regularity of the strong solution to \eqref{eq:caputo}
 can be improved  by looking at the equation satisfied by the weak derivative $y'$ and inspecting its smoothness. Since the focus of this paper lies on the optimal control and not on the analysis of fractional equations, we do not give a proof here. We just remark that the requirement $f(a(0)y_0+\ell(0))=0$ corresponds to the one in e.g.\,\cite[Thm\,6.26]{diethelm}, cf., also  \cite[Cor.\,2.2]{stynes}, where it is proven that the smoothness of the derivative of the strong solution improves if and only if such a compatibility condition is true.
\end{remark}

\section{Strong Stationarity}\label{sec:s}
The first result in this section will be an essential tool for establishing the strong stationarity in Theorem \ref{thm:st} below, as it guarantees the existence of a multiplier satisfying both a gradient equation and an inequality associated to a local minimizer of \eqref{eq:optt}.
\begin{lemma}\label{lem:tool}
Let  $(\bar a,\bar \ell) $ be a given local optimum of \eqref{eq:optt}. Then there exists a multiplier $\lambda \in L^r(0,T;\mathbb{R}^{n })$ with $r$ as in Lemma \ref{lem:int} such that
\begin{subequations}\label{eq:strongstat_q}
 \begin{gather}
 (\bar a,\delta a)_{H^1(0,T;\mathbb{R}^{n \times n})}+ {\dual{\lambda}{ \delta a \,\bar y}_{L^r(0,T;\mathbb{R}^{n})}}=0 \quad \forall\,\delta a \in {H^1(0,T;\mathbb{R}^{n \times n})},\label{eq:grad_qq}
\\[1mm]  (\bar \ell, \delta \ell)_{H^1(0,T;\mathbb{R}^{n})}+ {\dual{\lambda}{ \delta \ell}_{L^r(0,T;\mathbb{R}^{n })}} \geq 0 \quad \forall\, \delta \ell \in \cone (\KK-\bar \ell),\label{eq:grad_q}
 \end{gather}
 \end{subequations}where we abbreviate  $\bar y:=S(\bar a, \bar \ell).$
 \end{lemma}
 \begin{proof}
 The technical proof can be found in Appendix \ref{sec:a}.
 \end{proof}

The next step towards the derivation of our strong stationary system is to write the first order necessary optimality conditions in primal form.

\begin{lemma}[B-stationarity]
If $(\bar a,\bar \ell) $ is locally optimal for \eqref{eq:optt}, then there holds
\begin{multline}\label{eq:nec}
 {\nabla g(\bar y(T))^\top}\ S'\left((\bar a,\bar \ell); (\delta a, \delta \ell) \right)(T)
 \\
 + (\bar a,\delta a )_{H^1(0,T;\mathbb{R}^{n \times n})}+ (\bar \ell,\delta \ell )_{H^1(0,T;\mathbb{R}^n)} \geq 0
 \\
 \forall\, (\delta a , \delta \ell) \in H^1(0,T;\mathbb{R}^{n \times n})\times \cone (\KK-\bar \ell),
\end{multline}
 where we abbreviate $\bar y:=S(\bar a,\bar \ell)$.\end{lemma}
 \begin{proof}
The result  follows from the continuous differentiability of $g$ combined with the directional differentiability of $S$, see Theorem \ref{prop:dir},  and the local optimality of $(\bar a,\bar \ell)$.
\end{proof}

\begin{assumption}[`Constraint Qualification']\label{assu:cq}
There exists some index $m \in \{1,...,n\}$
such that the optimal state satisfies  $\bar y_m(t) \neq 0$
for all $t\in [0,T]$.\end{assumption}

\begin{remark}\label{rem:cq}
Let us underline that there is a zoo of situations where the requirement in Assumption \ref{assu:cq} is fulfilled. We just enumerate a few in what follows.
\begin{itemize}
\item If there exists some index  $m \in \{1,...,n\}$ such that $y_{0,m}>0$ and $f(z) \geq 0 \quad \forall\,z \in \mathbb{R}$ then
 the optimal state satisfies  $\bar y_m(t)\geq y_{0,m} > 0$ for all $t\in [0,T]$, in view of \eqref{eq:1}. {In particular, our `constraint qualification' is  fulfilled by continuous fractional deep neural networks (DNNs) with ReLU activation function, since $f=\max\{0,\cdot\}$ in this  case, while an additional initial datum can be chosen so that $y_{0,m}>0$}. \item Similarly, if there exists some index  $m \in \{1,...,n\}$ such that $y_{0,m}<0$ and $f(z) \leq 0 \quad \forall\,z \in \mathbb{R},$ then
 the optimal state satisfies  $\bar y_m(t)\leq  y_{0,m} < 0$ for all $t\in [0,T]$. In both situations, the CQ in Assumption \ref{assu:cq} is satisfied.
\item  If there exists some index  $m \in \{1,...,n\}$ such that  $y_{0,m} \neq 0$ and $f(\ll _m(t))=0$ for all $t\in [0,T]$, then, according to \cite[Thm.\,6.14]{diethelm}  the optimal state satisfies  $\bar y_m(t) \neq 0$ for all $t\in [0,T]$. This is the case if e.g.\, $f=\max\{0,\cdot\}$ and $\KK \subset \{v \in H^{1}\ot: v_m(t) \leq 0\  \forall\, t \in [0,T]\}$.
\end{itemize}\end{remark}

\begin{remark}\label{rem:cq0}
We point out that  Assumption \ref{assu:cq}
is due to the structure of the state equation and due to the fact that constraints are imposed on the control $\ell$ (and not on the control $a$), see Remark \ref{rem:surj} below for more details. The claim concerning the optimal state in Assumption \ref{assu:cq} is essential for deriving the strong stationary optimality system \eqref{eq:strongstat} below and it   plays the role of a `constraint qualification' (CQ), cf.\ e.g.\ \cite[Ch.\ 6]{troe}.
This terminology has its roots in finite-dimensional non-linear optimization, where it describes a condition for the (unknown) local optimizer which guarantees the existence of Lagrange multipliers such that a KKT-system is satisfied, see e.g.\ \cite[Sec.\ 2]{gk}. In the non-smooth case, the KKT conditions correspond to the strong stationary optimality conditions, see Remark \ref{rem:kkt} below.\end{remark}

The following result describes the density of the set of arguments into which the non-smoothness is derived in the ``linearized'' state equation \eqref{eq:2}. This aspect is crucial in the context of proving strong stationarity for the control of non-smooth equations, cf.\,\cite{st_coup,paper} and Remark \ref{rem:surj} below.
\begin{lemma}[Density of the set of arguments of $f'((\bar a \bar y+\bar \ell)_i(t);\cdot)$]\label{lem:dense}
Let  $(\bar a,\bar \ell) $ be a given local optimum for \eqref{eq:optt} with associated state $\bar y:=S(\bar a, \ll)$. Under Assumption \ref{assu:cq}, it holds
\[\{\bar a S'((\bar a,\ll);(\delta a,0))+\delta a \bar y:\delta a \in H^1\otm\} \dense \cont . \]
\end{lemma}\begin{proof}
Let $\rho \in \cont$ be  arbitrary, but fixed and define the  function
\[\widehat{\delta y}(t):=\int_0^t \frac{(t-s)^{\gamma-1}}{\Gamma(\gamma)} f '(\bar a(s)\bar y(s)+\bar \ell(s);\rho(s))\ds \quad \forall\, t \in [0,T].\]
Note that $\widehat{\delta y} \in C([0,T];\mathbb{R}^n),$ in view of Lemma \ref{lem:bound}.
{
We will now construct $\widehat{\delta a}$ such that
}
\begin{equation}\label{rho}
\bar a\widehat{\delta y}+\widehat{\delta a} \bar y=\rho.
\end{equation}
{
This is possible due to Assumption \ref{assu:cq}. Indeed,
for $j=1,...,n$ and $t\in [0,T]$ we can define
}
\[
\widehat{\delta a}_{jm}(t):=\frac{(\rho (t)- \bar a(t)\widehat{\delta y}(t))_j}{\bar y_m(t)}, \qquad \widehat{\delta a}_{ji}(t):=0, \quad \text{for }i \neq m.
\]
Note that $\widehat{\delta a}_{jm} \in C[0,T]$.
{
Due to \eqref{rho}, $\widehat{\delta y}$ satisfies the integral equation
}
\begin{equation*}\label{eq:22}
\widehat{\delta y}(t)=\int_0^t \frac{(t-s)^{\gamma-1}}{\Gamma(\gamma)} f '(\bar a(s)\bar y(s)+\bar \ell(s);\bar a(s)\widehat{\delta y}(s)+\widehat{\delta a}(s) \bar y(s))\ds \quad \forall\, t \in [0,T].\end{equation*}
By Theorem \ref{prop:dir}, the integral equation is equivalent to
\begin{equation}\label{hat}
\widehat{\delta y}=S'((\bar a,\bar \ell);(\widehat{\delta a},0)).
\end{equation}
 Now, let us consider a sequence $\delta a_k \in H^1(0,T;\mathbb{R}^{n \times n})$ with
\begin{equation}\label{eq:conv_an}
\delta a_k \to \widehat{\delta a} \quad \text{ in }\contm.
\end{equation}
In view of Proposition \ref{prop:lip} and Theorem \ref{prop:dir},
the mapping $S$ is locally Lipschitz continuous and directionally differentiable from $\lia$ to $\cont$.
Hence,
the mapping $S'((\bar a,\bar \ell);\cdot):\lia \to \cont$ is continuous, see, e.g., \cite[Lem.\,3.1.2 and Lem.\,3.1.3]{schirotzek}. Thus, the convergence \eqref{eq:conv_an} implies that
\begin{equation*}\label{eq:conv_dyn}
 {\delta y}_k:=S'((\bar a,\bar \ell);(\delta a_k,0)) \to \widehat{ \delta y} \quad \text{ in }C([0,T];\mathbb{R}^n),
\end{equation*}where we recall \eqref{hat}. This gives in turn
\[\bar  a {\delta y}_k+\delta a_k \bar  y \to \rho \quad \text{ in }\cont,\]in view of \eqref{rho}. Since $\rho \in \cont$ was arbitrary, the proof is now complete.
\end{proof}

The main finding of this paper is stated in the following result.
\begin{theorem}[Strong stationarity]\label{thm:st}
Suppose that Assumption \ref{assu:cq} is satisfied and let  $(\bar a,\bar \ell) $ be a given local optimum for \eqref{eq:optt} with associated state $\bar y:=S(\bar a, \ll)$. Then there exists a  multiplier $\lambda \in L^r(0,T;\mathbb{R}^{n })$ and an adjoint state $p \in L^r(0,T;\mathbb{R}^{n })$, where $r$ is given by Lemma \ref{lem:int}, such that
\begin{subequations}\label{eq:strongstat}
 \begin{gather}
  p(t)= \frac{(T-t)^{\gamma-1}}{\Gamma(\gamma)} \nabla g (\bar y(T))+\int_t^T \frac{(s-t)^{\gamma-1}}{\Gamma(\gamma)} \bar a^\top(s)\lambda (s)\ds \quad \forall\, t \in [0,T).\label{eq:p2}
\\
\lambda_i(t)\in  [p_i(t)f_-' ((\bar a \bar y+\bar \ell)_i(t)),p_i(t)f_+' ((\bar a \bar y+\bar \ell)_i(t))]\quad \text{\ for a.a.\ }t \in (0,T),\ i=1,...,n,\label{sign_c}
\\
(\bar a,\delta a)_{H^1(0,T;\mathbb{R}^{n \times n})}+{\dual{\lambda}{ \delta a \,\bar y}_{L^r(0,T;\mathbb{R}^{n})}}=0 \quad \forall\,\delta a \in {H^1(0,T;\mathbb{R}^{n \times n})},\label{eq:grad_1}
\\
(\bar \ell, \delta \ell)_{H^1(0,T;\mathbb{R}^{n})}+ {\dual{\lambda}{ \delta \ell}_{L^r(0,T;\mathbb{R}^{n })} } \geq 0 \quad \forall\, \delta \ell \in \cone (\KK-\bar \ell),\label{eq:grad_2}
 \end{gather}
 \end{subequations}where, for an arbitrary $z\in \mathbb{R}$, the left and right-sided derivative of $f: \mathbb{R} \to \mathbb{R}$ are defined through $f'_+(z) := f'(z;1)$ and $ f'_-(z) := - f'(z;-1)$, respectively. \end{theorem}
\begin{remark}
The adjoint (integral) equation \eqref{eq:p2} describes the mild solution of
a differential equation featuring the so-called right Riemann-Liouville operator \cite[Ch.\ 5]{diethelm}:
\begin{equation}\label{r-l}
D^{\gamma}_{T-} (p)(t)=  \bar a^\top(t)\lambda (t) \quad \text{f.a.a.\, }t \in (0,T), \quad \lim_{t \to T}I^{1-\gamma}_{T-} (p)(t)=\nabla g (\bar y(T)),
\end{equation}where \[D^{\gamma}_{T-}  (\phi):=-\frac{d}{dt} I^{1-\gamma}_{T-} (\phi). \]
Here we recall Definition \ref{def:io}. If  $p$ is absolutely continuous, then, together with $\delta y=S'((\bar a,\ll);\dal) $, it satisfies  the relation in\,\cite[Prop.\,2.5]{unified}, which says  that the right Riemann-Liouville operator is the adjoint of the Caputo fractional derivative (Definition \ref{def:caputo}). Note that $\delta y \in W^{1,1}(0,T;\mathbb{R}^n)$; this can be  shown by arguing as in the proof of Theorem \ref{prop:w11}. If $p$ has enough regularity, then $I^{1-\gamma}_{T-} p \in \cont$ and thus, $I^{1-\gamma}_{T-} (p)(T)=\nabla g (\bar y(T)) $, in view of \eqref{r-l}.
\end{remark}

\begin{proof}[Proof of Theorem \ref{thm:st}]
{
From Lemma \ref{lem:tool}, we get the existence of $\lambda\in L^r(0,T;\mathbb{R}^{n })$ satisfying \eqref{eq:strongstat_q}.
}
{This allows us to define an adjoint state $p \in L^r(0,T;\mathbb{R}^{n })$  such that \eqref{eq:p2}, \eqref{eq:grad_1} and \eqref{eq:grad_2} are satisfied.
Note that the $ L^r(0,T;\mathbb{R}^{n })$ regularity of $p$ is a result of Lemmas \ref{lem:int} and \ref{lem:bound}.}
Thus, it remains to show that  \eqref{sign_c} is true.
%To this end, we employ arguments similar to those used in the proof of \cite[Thm. 2.11]{st_coup}.
Let $(\delta a, \delta \ell) \in H^1(0,T;\mathbb{R}^{n \times n}) \times  \cone (\KK-\bar \ell)$ be arbitrary but fixed  and  abbreviate $\widetilde{ \delta y}:=S'((\bar a,\bar \ell);(\delta a,\dl))$ and $f ' (\cdot;\cdot):=f ' (\bar a \bar y+\ll;\bar a \widetilde{ \delta y}+ \delta a \bar  y +\dl)$. Note
that \begin{equation}\label{eq:deltayy}
\widetilde{\delta y} =I_{0+}^{\gamma}(f ' (\cdot;\cdot)),
\end{equation}
 see \eqref{eq:2}.
Now, using \eqref{eq:p2} and \eqref{eq:deltayy} in Lemma \ref{fip} leads to
  \[\int_0^T  {(\bar a^\top\lambda)(t)^\top} \widetilde{ \delta y}(t) \dt =\int_0^T  {f ' (\cdot;\cdot)(t)^\top} [p(t)-\frac{(T-t)^{\gamma-1}}{\Gamma(\gamma)}\nabla g(\bar  y(T))]\dt.\] Thus,
\begin{equation}\label{eqq}
\int_0^T  {f ' (\cdot;\cdot)(t)^\top} p(t)- {(\bar a^\top\lambda)(t)^\top} \widetilde{ \delta y}(t) \dt= {\nabla g(\bar  y(T))^\top} \underbrace{\int_0^T \frac{(T-t)^{\gamma-1}}{\Gamma(\gamma)} f ' (\cdot;\cdot)(t) \dt }_{=\widetilde{ \delta y}(T), \text{ see }\eqref{eq:deltayy}}.
\end{equation}
By inserting  \eqref{eqq}  in \eqref{eq:nec}, we arrive at
\begin{multline}\label{eq:st1}
 \int_0^T {f ' (\cdot;\cdot)(t)^\top} p(t)-{(\bar a^\top\lambda)(t)^\top} \widetilde{ \delta y}(t) \dt
\\
+ (\bar a,\delta a)_{H^1\otm}
+ (\ll,\delta \ell)_{H^1\ot} \geq 0
\\
\forall\, (\delta a,\dl)  \in H^1(0,T;\mathbb{R}^{n \times n})\times  \cone (\KK-\bar \ell).
\end{multline}
Setting $\dl:=0$, taking into account $\widetilde{ \delta y}=S'((\bar a,\bar \ell);(\delta a,0))$ and the definition of $f ' (\cdot;\cdot)$,
and  making use of \eqref{eq:grad_1}, results in
\begin{multline}\label{eq:st}
\int_0^T   {f ' (\cdot;\cdot)(t)^\top} p(t)-  {(\bar a^\top\lambda)(t)^\top} \widetilde{ \delta y}(t) \dt+ (\bar a,\delta a)_{H^1\otm}
\\
\begin{aligned}
\overset{\eqref{eq:grad_1}}{=} & \int_0^T   {p(t)^\top} f ' (\bar a(t) \bar y(t)+\ll(t);\bar a(t) \widetilde{ \delta y}(t)+\delta a(t) \bar  y(t))  {\dt}
\\
&
-   {\int_0^T \lambda(t)^\top}(\bar  a(t) \widetilde{ \delta y}(t)+\delta a(t) \bar  y(t)) \dt
\geq 0 \quad \forall\, \delta a  \in H^1(0,T;\mathbb{R}^{n \times n}).
\end{aligned}
\end{multline}
Now let $\rho \in \cont$ be arbitrary but fixed. According to Lemma \ref{lem:dense} there exists a sequence $\{\delta a_n\} \subset H^1 \otm $ such that
\[\bar a S'((\bar a,\bar \ell);(\delta a_n,0)) +\delta a_n \bar  y \to \rho \quad \text{ in }\cont.\]
Thus, testing with $\delta a_n \in H^1 \otm $ in \eqref{eq:st} and passing to the limit $n \to \infty$ leads to
\begin{equation*}\label{eq:tst}
\int_0^T   {p(t)^\top}f ' (\bar a(t) \bar y(t)+\ll(t);\rho(t))-   {\lambda(t)^\top}\rho(t) \dt
\geq 0 \quad \forall\, \rho \in C ([0,T];\mathbb{R}^n),
 \end{equation*}
 where we relied on the continuity of $f ' (\bar a \bar y+\ll;\cdot):\liv \to \liv$, cf.\,\eqref{f_dir}, and on the fact that $\lambda, p \in L^r \ot$.
Now, by testing with $\rho \geq 0$ and by employing the
 fundamental lemma of calculus of variations combined with  the positive homogeneity of the directional derivative with respect to the direction we deduce
 \begin{align*}
p_i(t)f ' ((\bar a \bar y+\bar \ell)_i(t);1)-\lambda_i(t)
\geq 0 \qquad \text{a.e. in }(0,T), \quad i=1,...n.
 \end{align*}
In an analogous way, testing with $\rho \leq 0$ implies
 \begin{align*}
p_i(t)f ' ((\bar a \bar y+\bar \ell)_i(t);-1)+\lambda_i(t)
\geq 0 \qquad  \text{a.e. in }(0,T), \quad i=1,...n,
 \end{align*}
from which  \eqref{sign_c} follows.
\end{proof}

\begin{remark}[Correspondence to KKT conditions]\label{rem:kkt}
If $(\bar a \bar y+\bar \ell)_i(t) \not \in \NN$ f.a.a.\ $t \in (0,T)$ and for all $i=1,...,n$, where $\NN$ denotes the set of non-smooth points of $f$, then $\lambda_i(t) =p_i(t)f ' ((\bar a \bar y+\bar \ell)_i(t))$ f.a.a.\ $t \in (0,T)$ and for all $i=1,...,n$, cf.\,\eqref{sign_c}. In this case, \eqref{eq:strongstat_q} reduces to the standard KKT-conditions.\end{remark}

The optimality system in Theorem \ref{thm:st} is indeed of strong stationary type, as the next result shows:

\begin{theorem}[Equivalence between B- and strong stationarity]\label{thm:B_qsep}
Let  $(\bar a,\bar \ell) \in H^1(0,T;\mathbb{R}^{n \times n})\times \KK$ be given and let  $\bar y:=S(\bar a, \ll)$ be its associated state. If there exists a  multiplier $\lambda \in L^r(0,T;\mathbb{R}^{n })$ and an adjoint state $p \in L^r(0,T;\mathbb{R}^{n })$, where $r$ is given by Lemma \ref{lem:int}, such that
\eqref{eq:strongstat} is satisfied, then
$(\bar a,\bar \ell) $ also satisfies the variational inequality \eqref{eq:nec}. Moreover, if Assumption \ref{assu:cq} is satisfied, then \eqref{eq:nec} is equivalent to \eqref{eq:strongstat}.
\end{theorem}
\begin{proof}
We first show that \eqref{sign_c} implies
\begin{equation}\label{eq:tst1}
\int_0^T {p(t)^\top}f ' (\bar a(t) \bar y(t)+\ll(t);\rho(t))- {\lambda(t)^\top}\rho(t) \dt
\geq 0  \quad \forall\, \rho \in C ([0,T];\mathbb{R}^n).
 \end{equation}To this end, let
$\rho \in \cont$ and $i=1,...,n$ be arbitrary, but fixed. We denote by $\NN$ the set of non-differentiable points of $f$.
  From  \eqref{sign_c}, we deduce that
  \begin{equation} \label{eq:da} \begin{aligned}
\lambda_i(t)\rho_i(t) =p_i(t) f '((\bar a \bar y+\ll)_i(t))\rho_i(t) \quad \text{a.e.\ where } (\bar a \bar y+\ll)_i \not\in \NN.
\end{aligned} \end{equation}
Further, we define $\NN_i^+:=\{t \in [0,T]:(\bar a \bar y+\ll)_i(t)\in \NN \text{ and } \rho_i(t)>0 \}$ and $\NN_i^-:=\{t \in [0,T]:(\bar a \bar y+\ll)_i(t) \in \NN \text{ and } \rho_i(t)\leq 0 \}$.   Then,  \eqref{sign_c} and the positive homogeneity of the directional derivative with respect to the direction yield
\begin{equation}  \label{eq:db} \begin{aligned}
\lambda_i(t)\rho_i(t) &\leq \left \{
\begin{aligned}
 &  p_i(t)f_+' ((\bar a \bar y+\bar \ell)_i(t))\rho_i(t)
&& \ \ \text{ a.e.\ in  }\NN_{i}^+
\\ &p_i(t)f_-' ((\bar a \bar y+\bar \ell)_i(t))\rho_i(t) && \ \  \text{ a.e.\ in  }\NN_{i}^-
\end{aligned} \right.
\\&\quad =p_i(t) f '((\bar a \bar y+\ll)_i(t);\rho_i(t) )\quad \text{a.e.\ where } (\bar a \bar y+\ll)_i\in \NN.
\end{aligned}
\end{equation}
Now, \eqref{eq:tst1} follows from \eqref{eq:da} and \eqref{eq:db}.

Next, let $(\delta a,\dl) \in H^1(0,T;\mathbb{R}^{n \times n}) \times {\cone(\KK-\bar \ell)}$ be arbitrary but fixed and test \eqref{eq:tst1} with $\bar a \widetilde{\delta y}+  {\delta a \bar  y +\dl},$ where  we abbreviate $\widetilde{\delta y}:=S'((\bar a,\bar \ell);(  {\delta a,\dl}))$. This results in
\begin{equation}\label{eq:sst}
\begin{aligned}
&\int_0^T   {p(t)^\top f ' (\bar a(t) \bar y(t)+\ll(t);(\bar a \widetilde{\delta y}+\delta a  \bar  y +\dl)(t))\dt}
\\&\qquad -  {\int_0^T\lambda(t)^\top (\bar a \widetilde{\delta y}+
\delta a
 \bar  y +\dl)(t)\dt \geq 0 .}
 \end{aligned}
 \end{equation}Then, by using \eqref{eqq} one sees that \eqref{eq:sst} implies
\begin{multline}\label{eq:necc1}
  {\nabla g (\bar y(T))^\top \widetilde {\delta y}(T)} - {\dual{\lambda}{\delta a  \bar  y +\dl}_{L^r(0,T;\mathbb{R}^n)} } \geq 0
  \\
  \forall\, (\delta a,\dl)  \in H^1(0,T;\mathbb{R}^{n \times n}) \times {\cone(\KK-\ll)}.
\end{multline}
Finally \eqref{eq:grad_1}-\eqref{eq:grad_2} in combination with \eqref{eq:necc1} yield that \eqref{eq:nec} is true. Moreover, if Assumption \ref{assu:cq} is satisfied, then \eqref{eq:nec} implies \eqref{eq:strongstat}, see the proof of Theorem \ref{thm:st}. We underline that the only information about the local minimizer that is used in the proof of Theorem \ref{thm:st} is contained in \eqref{eq:nec}.
\end{proof}

 \begin{remark}[Strong stationarity in the case $\gamma=1$]
 If $\gamma =1$, then the state $\bar y$ associated to the local optimum $(\bar a,\bar \ell) \in H^1 \ota$ belongs to $W^{2,2} \ot$; this is a consequence of the statement in Remark \ref{rem:1} combined with the fact that $f(\bar a\bar y+\bar \ell) \in H^{1}(0,T;\mathbb{R}^n)$, since $f \in W^{1,\infty}(\mathbb{R}^n;\mathbb{R}^n)$, as a result of Assumption \ref{assu:stand}.\ref{it:1}. Moreover, by taking a look at \eqref{eq:p2} we see that the adjoint equation reads
  \begin{equation*}
  -p'(t)= \bar a^\top(t)\lambda (t) \quad \forall\, t \in [0,T], \quad p(T)=\nabla g (\bar y(T))
  \end{equation*}for $\gamma=1$. A close inspection of step (III) in the proof of Lemma \ref{lem:tool} shows that $p \in W^{2,2} \ot$ and $\lambda \in \liv$, see \eqref{sign_c}.
 \end{remark}
\subsection*{Some comments regarding the main result}
We end this section by collecting some important remarks concerning Theorem \ref{thm:st}.

\begin{remark}[Density of the set of arguments of $f'((\bar a \bar y+\bar \ell)_i(t);\cdot)$]\label{rem:surj}
The proof of Theorem \ref{thm:st} shows that it is essential that the set of directions  {into }which the non-smooth mapping $f$ is {differentiated} - in the `linearized' state equation associated to $(\bar a,\ll)$ -
is dense in a (suitable) Bochner space (which is  the assertion in Lemma \ref{lem:dense}). This has also been pointed out in \cite[Rem.\,2.12]{st_coup}, where strong stationarity for a coupled non-smooth system is proven.

Let us underline that the `constraint qualification' in Assumption \ref{assu:cq}
is not only due  to the structure of the state equation, but also due to the presence of constraints on $\ell$.
If constraints were imposed on $a$ instead of $\ell$, then there would be no need for a CQ in the sense of Assumption \ref{assu:cq}. An inspection of the proof of Theorem \ref{thm:st} shows that in this case one needs to show that
\[\{\bar a S'((\bar a,\ll);(0,\delta \ell))+\delta \ell:\delta \ell \in H^1\ot\} \dense \cont . \]This is done by arguing as in the proof of Lemma \ref{lem:dense}, where this time, one defines $\widehat \delta \ell:=\rho-\bar a \widehat \delta y.$

Thus, depending on the setting, the  `constraint qualification' may vanish or may read completely differently \cite[Assump.\,2.6]{st_coup}, but it should imply that the set of directions into which $f $ is differentiated -in the ``linearized'' state equation-is dense in an entire space \cite[Lem.\,2.8]{st_coup}, see also \cite[Rem.\,2.12]{st_coup}.

These observations are also consistent with
the result in \cite{paper}. Therein, the direction {into} which one {differentiates} the non-smoothness - in the `linearized' state equation -  is the `linearized' solution operator, such that the counterpart of our Lemma \ref{lem:dense}  is \cite[Lem.\ 5.2]{paper}.
 In \cite{paper}, there is no constraint qualification in the sense of Assumption \ref{assu:cq}; however, the density assumption \cite[Assump. 2.1.6]{paper} can be regarded as such.
In \cite[Rem.\,4.15]{cr_meyer} the authors also acknowledge the necessity of a density condition similar to that described above in order to ensure strong stationarity.
\end{remark}

\begin{remark}[Control constraints]\label{rem:cc}
We point out that we deal with controls $(a,\ell)$ mapping to $(\mathbb{R}^{n})^{n+1}$, whereas the space of functions we want to cover in Lemma \ref{lem:dense} consists of functions that map to $\mathbb{R}^n$ only. This allows us to restrict $n$ controls by constraints (if we look at \eqref{eq:optt} as having $n+1$ controls mapping to $\mathbb{R}^n$.) Indeed, a closer inspection of  the proof of Lemma \ref{lem:dense} shows that one can impose control constraints  on all  columns of the control $a$ {except} the $m-$th column.
This still  implies that
the set of directions {into }which  $f$ is {differentiated} - in the `linearized' state equation - {is dense in} an entire space.
%When it comes to control constraints in the context of strong stationarity, other papers require additional constraint qualifications \cite{wachsm_2014, mcrf}.
The fact that two or more controls provide advantages in the context of strong stationarity has already been observed in \cite[Sec. 4]{hmw13}. Therein, an additional control has to be considered on the right-hand side of the  VI under consideration  in order to be able to prove strong stationarity, see \cite[Sec. 4]{hmw13} for more details.

%Note that our 'constraint qualification' in Assumption \ref{assu:cq} is due to the structure of the state equation and has nothing to do with the fact that $\ell$ is restricted by constraints.(\ko{actually it does})
The situation changes when, in addition to asking that $\ell \in \KK$, control constraints are imposed on all columns of $a$. In this case, we deal with a fully  control  constrained problem. By looking at the proof of Lemma \ref{lem:dense} we see  that the arguments cannot be applied in this case, see also \cite{mp84, cr_meyer, paper, st_coup} where the same observation was made.
This calls for a different approach in the proof of Theorem \ref{thm:st} and additional ``constraint qualifications'' \cite{wachsm_2014, mcrf}.
\end{remark}

\begin{remark}[Sign condition on the adjoint state. Optimality conditions obtained via smoothening]\label{rem:sign}

(i) An essential information contained in  the strong stationary system \eqref{eq:strongstat} is the fact that \begin{equation}\label{sign}
p_i(t)f_-' ((\bar a \bar y+\bar \ell)_i(t)) \leq p_i(t)f_+' ((\bar a \bar y+\bar \ell)_i(t))
\end{equation}
f.a.a.\ $t \in (0,T),\ i=1,...,n$, see \eqref{sign_c}. This is crucial for showing the implication $\eqref{eq:strongstat} \Rightarrow \eqref{eq:nec}$, which ultimately yields that \eqref{eq:strongstat} is indeed of strong stationary type (see  the proof of Theorem \ref{thm:B_qsep}).

If $f$ is convex or concave around its {non-smooth points}, this translates into a sign condition for the adjoint state. Indeed, if $f:\mathbb{R} \to \mathbb{R}$ is convex around a non-smooth point $z$, this means that $f_-' (z) < f_+' (z) $, and  from \eqref{sign} we have
\begin{equation*}\label{eq:sign}
p_i(t)\geq 0 \quad \text{a.e.\ where }(\bar a \bar y+\bar \ell)_i=z ,\ \  i=1,...,n.\end{equation*}Similarly, in the concave case, $p$ is negative for those pairs $(i,t)$ for which $(\bar a \bar y+\ll)_i(t)$ is  a non-differentiable point of $f$.

{In addition, we note that, if $f$ is piecewise continuously differentiable, \eqref{sign}  implies
 the regularity (cf.\ \cite[Def.\ 7.4.1]{schirotzek}) of the mapping $p_i(t)f:\mathbb{R} \to \mathbb{R}$ at $(\bar a \bar y+\ll)_i(t)$ f.a.a.\,$t \in (0,T)$ and for all $i=1,...,n$, in view of  \cite[Lem.\ C.1]{paper}. See also\, \cite[Rem.\ 6.9]{paper} and \cite{st_coup} for  similar situations.}

(ii) By contrast, optimality systems derived by classical smoothening techniques often lack  a sign  for the adjoint state (and the above mentioned regularity in the sense of \cite[Def.\ 7.4.1]{schirotzek}) eventually along with other information which gets lost in the limit analysis.  See e.g.\ \cite[Prop.\,2.17]{mcrf}, \cite[Sec.\ 4]{paper}, \cite[Thm.\ 4.4]{cr_meyer}, \cite[Thm.\ 2.4]{tiba} (optimal control of non-smooth PDEs) and \cite{mp84}  (optimal control of VIs).
Generally speaking, a sign condition for the adjoint state in those points $(i,t)$ where the argument of  the non-smoothness $f$ in the state equation, in our case $(\bar a \bar y+\ll)_i(t)$, is such that $f$ is not differentiable at $(\bar a \bar y+\ll)_i(t)$, is what ultimately distinguishes a strong stationary optimality system from very `good'    optimality systems obtained by smoothening procedures, cf.\  \cite[Prop.\,2.17]{mcrf} and   \cite[Sec.\ 7.2]{paper}, see also \cite[Rem.\,4.15]{cr_meyer}.      \end{remark}

\appendix
\section{Proof of Lemma \ref{lem:tool}}
\label{sec:a}

\begin{proof}[Proof of Lemma \ref{lem:tool}]
{We associate a state equation to a smooth approximation of the non-differentiable function ${f}$, such that the respective solution mapping is G\^{a}teaux-differentiable (step (I)). }Then, by arguments inspired by e.g.\ \cite{barbu}, it follows that $(\bar a,\ll)$ can be approximated by a sequence of local minimizers of an optimal control problem governed by the regularized state equation (step (II)). Passing to the limit in the adjoint system associated to the regularized optimal control problem finally yields the desired assertion (step (III)). Although many of the arguments are well-known, we give a detailed proof, for completeness and for convenience of the reader.

 (I) Let $\varepsilon>0$ be arbitrary, but fixed.
We begin by investigating the smooth  integral equation
\begin{equation}\label{eq:smooth}
y_\eps(t)=y_0+\int_0^t \frac{(t-s)^{\gamma-1}}{\Gamma(\gamma)} f{_\varepsilon} (a(s)y_\eps(s)+\ell(s))\ds \quad \forall\, t \in [0,T],
 \end{equation}
 where the differentiable function $f_\eps:\mathbb{R} \to \mathbb{R}$ is defined  as \[f_\eps(z):=\int_{-\infty}^{\infty} f(z-\eps s)\varphi(s)\ds,\]
 where $\varphi \in C_c^\infty(\mathbb{R}),\ \varphi \geq 0, \ \supp \varphi \subset [-1,1]$ and $\int_{-\infty}^{\infty} \varphi(s)\ds=1$. Once again, we do not distinguish between $f_\eps:\mathbb{R}^n \to \mathbb{R}^n$ and $f_\eps:\mathbb{R} \to \mathbb{R}$. As in the case of its non-smooth counterpart, $f_\eps:\mathbb{R}^n \to \mathbb{R}^n$ is assumed to satisfy
 for all $z \in \mathbb{R}^n$
\begin{equation*}
f_\eps(z)_i=\widetilde f_\eps(z_i), \quad i=1,...,n
\end{equation*}
 where $\widetilde f_\eps:\mathbb{R} \to \mathbb{R}$ is a smooth function. We observe that
for all $z \in \mathbb{R}$ it holds
  \begin{equation}\label{it:feps1}
f_\varepsilon(z) \to f(z) \quad \text{as } \varepsilon \searrow 0.
  \end{equation}Moreover,
   \begin{equation}\label{it:feps2}
\|f_\varepsilon(z_1)-f_\varepsilon(z_2) \|\leq L\,\|z_1-z_2\| \quad \forall\,z_1,z_2 \in \mathbb{R}^n,  \end{equation}
where $L>0$ is the Lipschitz constant of $f$.

  By employing the exact same arguments as in the proof of Proposition \ref{prop:existence}, one infers that \eqref{eq:smooth} admits a unique solution $y_\eps \in \cont$ for every  $(a,\ell) \in \lia$, which allows us to define the smooth solution mapping \[S_\varepsilon:\lia \ni (a,\ell) \mapsto y_\eps \in  C([0,T];\mathbb{R}^n).\] The operator $S_{\varepsilon}$ is G\^{a}teaux-differentiable and its derivative
  is the unique solution of
 \begin{equation}\label{eq:ode_lin_q_e}
\delta y_\eps(t)=\int_0^t \frac{(t-s)^{\gamma-1}}{\Gamma(\gamma)} f{_\varepsilon} '(a(s)y_\eps(s)+\ell(s))(a(s)\delta y_\eps(s)+\delta a(s) y_\eps(s)+\delta \ell(s))\ds\end{equation}for all $t \in [0,T]$, i.e., $\delta y_\eps=S_\eps'(a,\ell)(\delta a, \delta \ell)$; note that here we use the notation $f{_\varepsilon}'$ for the Jacobi-matrix of $f{_\varepsilon}:\mathbb{R}^n \to \mathbb{R}^n$.
By using the integral formulations \eqref{eq:1} and \eqref{eq:smooth}, Lemma \ref{lem:g} and \eqref{it:feps1}, we obtain the convergence $  S_\varepsilon \al-S(a, \ell) \to 0  \  \text{in } C([0,T];\mathbb{R}^n)
$ as $\eps \searrow 0$. On the other hand, by arguing as in the proof of Proposition \ref{prop:lip} we deduce that $S_\varepsilon$ is Lipschitz-continuous in the sense of  \eqref{eq:loclipS} (with constant independent of $\varepsilon$).
 As a result, we have
  \begin{equation}\label{s_conv}
  S_\varepsilon \ale-S(a, \ell) \to 0  \quad \text{in } C([0,T];\mathbb{R}^n),
  \end{equation}when $\ale \to \al$ in $\lia$.

(II)  Next, we focus on proving that $(\bar a,\ll)$ can be approximated via local minimizers of optimal control problems governed by \eqref{eq:smooth}. To this end, let
\[
B_\rho:=B_{H^1\ota}((\bar a,\ll),  \rho) \cap (H^1\otm \times \KK), \ \rho>0, \] be the ball of local optimality of $(\bar a,\ll)$ and consider the smooth (reduced) optimal control problem
\begin{equation}\label{eq:q_eps}
 \left.
 \begin{aligned}
  \min_{(a,\ell) \in H^1\ota} \quad & J(S_{\eps}(a,\ell),a,\ell)+\frac{1}{2}\|(a,\ell)-(\bar a,\ll)\|_{H^1\ota}^2
  \\
 \text{s.t.} \quad & (a,\ell) \in B_\rho.
 \end{aligned}
 \quad \right\}
\end{equation}  Let us recall here that
\[J(y,a,\ell)=g(y(T)) +\frac{1}{2} \|a\|^2_{H^1(0,T;\mathbb{R}^{n \times n})} +\frac{1}{2} \|\ell\|^2_{H^1(0,T;\mathbb{R}^{n })},\]
where  $g: \mathbb{R}^n \to \mathbb{R}$ is a differentiable, and thus, continuous function.
By the Lipschitz continuity of $S_\eps:L^{\infty}{\ota} \to C([0,T];\mathbb{R}^n)$ and the compact embedding $H^1\ota \embed \embed L^{\infty}\ota$,
we see that \eqref{eq:q_eps} admits a global solution $(a_\eps,\ell_{\eps}) \in B_\rho$. Since $B_\rho$ is weakly closed in $H^1\ota$ we can extract a weakly convergent subsequence
\[(a_\eps,\ell_\eps) \weakly (\widetilde a, \widetilde \ell)\in B_\rho \quad \text{ in }H^1 \ota.\]
For simplicity, we abbreviate in the following   \begin{subequations} \begin{gather}
\JJ(a,\ell):=J(S(a,\ell),a,\ell) ,\label{jj}\\
\JJ_{\eps}(a,\ell):=J(S_{\eps}(a,\ell),a,\ell)+\frac{1}{2}\|(a,\ell)-(\bar a,\ll)\|_{H^1\ota}^2  .\label{jn}
\end{gather}\end{subequations}
Due to \eqref{s_conv} combined with the continuity of $g:\mathbb{R}^n \to \mathbb{R}$,  it holds
\begin{multline}\label{j}
  \JJ(\bar a, \ll)\overset{\eqref{jj}}{=}J(S(\bar a, \ll),\bar a, \ll )=\lim_{\eps \to 0} J(S_{\eps}(\bar a ,\ll),\bar a , \ll)
  \\
  \overset{\eqref{jn}}{=}\lim_{\eps \to 0}  \JJ_{\eps}(\bar a, \ll) \geq \limsup_{\eps \to 0}   \JJ_{\eps}(a_\eps,\ell_{\eps}),
\end{multline}
where in the last inequality we used the fact that $(a_{\eps}, \ell_{\eps})$ is a global minimizer of \eqref{eq:q_eps} and that $(\bar a, \ll)$ is admissible for \eqref{eq:q_eps}. In view of  \eqref{jn}, \eqref{j} can be continued as
   \begin{equation}\label{j1}\begin{aligned}
  \JJ(\bar a, \ll) &\geq \limsup_{\eps \to 0}   J(S_{\eps}(a_{\eps}, \ell_{\eps}),a_{\eps}, \ell_{\eps})+\frac{1}{2}\|(a_{\eps}, \ell_{\eps})-(\bar a, \ll)\|_{H^1\ota}^2
   \\&\quad   \geq \liminf_{\eps \to 0}   J(S_{\eps}(a_{\eps}, \ell_{\eps}),a_{\eps}, \ell_{\eps})+\frac{1}{2}\|(a_{\eps}, \ell_{\eps})-(\bar a, \ll)\|_{H^1\ota}^2
      \\&\qquad   \geq   J(S\alw,\widetilde a, \widetilde \ell)   +\frac{1}{2}\|\alw-(\bar a, \ll)\|_{H^1\ota}^2   \geq   \JJ(\bar a, \ll),
\end{aligned}  \end{equation}
where we used again $H^1\ota \embed \embed \li$ and \eqref{s_conv} combined with the continuity of $g:\mathbb{R}^n \to \mathbb{R}$; note that for the last inequality in \eqref{j1} we employed the fact that $(\widetilde a, \widetilde \ell)\in B_\rho$.
From \eqref{j1} we conclude \[\lim_{\eps \to 0}   J(S_{\eps}(a_{\eps}, \ell_{\eps}),a_{\eps}, \ell_{\eps})+\frac{1}{2}\|(a_{\eps}, \ell_{\eps})-(\bar a, \ll)\|_{H^1\ota}^2 =\JJ(\bar a, \ll).
\]
By arguing as above we also get
\[\lim_{\eps \to 0}   J(S_{\eps}(a_{\eps}, \ell_{\eps}),a_{\eps}, \ell_{\eps})=\JJ(\bar a, \ll),
\]which implies
\begin{equation}\label{al_conv}
(a_{\eps}, \ell_{\eps}) \to (\bar a, \ll)  \quad \text{in }H^1\ota.
  \end{equation}
 As a consequence, \eqref{s_conv} yields
\begin{equation}\label{y_conv}
  y_{\varepsilon} \to \bar y  \quad \text{in } C([0,T];\mathbb{R}^n),
  \end{equation}where we abbreviate $ y_\eps:=S_{\eps}(a_{\eps}, \ell_{\eps}).$
By classical arguments  one then obtains that $\ale$ is a local minimizer for 
\[\min_{ H^1\otm \times \KK} \JJ_\eps \al.\]

(III)   Due to the above established  local optimality of $(a_{\eps}, \ell_{\eps})$ and on account of the differentiability properties of $S_{\eps}$, cf.\ step (I),  we can write down the following necessary optimality condition
 \begin{equation}\label{necc}\begin{aligned}
{ \nabla g(y_{\eps}(T))^{\top}}&S_{\eps}'(a_{\eps}, \ell_{\eps})(\al-\ale)(T)
 \\&\qquad \quad +(2a_{\eps}-\bar a,a-a_\eps)_{H^1(0,T;\mathbb{R}^{n \times n})}
 \\&\quad \qquad \quad +(2 \ell_{\eps}- \ll,\ell-\ell_\eps)_{H^1(0,T;\mathbb{R}^n)} \geq 0
\end{aligned}\end{equation} for all $\al \in H^1(0,T;\mathbb{R}^{n \times n}) \times \KK$.
%\ko{By setting $\da:=+/-\da +a_\eps$ one has
% \begin{equation}\label{nec11}\begin{aligned}
% (y_{\eps}(T)-&y_d,S_{\eps}'(a_{\eps}, \ell_{\eps})(\da,\dl)(T))+(2a_{\eps}-\bar a,\da)_{H^1(0,T)} +(2\ell_{\eps}- \ll, \dl)_{H^1(0,T)} \geq 0 \\&\quad \forall\, \dal \in \hs \times  \cone (\KK-\ell_\eps) .
%\end{aligned}\end{equation} }
% \begin{equation}\label{nec12}\begin{aligned}
% (y_{\eps}(T)-y_d,S_{\eps}'(a_{\eps}, \ell_{\eps})(0,\dl)(T))+(2\ell_{\eps}- \ll, \dl)_{H^1(0,T)} \geq 0 \quad \forall\, \dl \in \cone (\KK-\ll).
%\end{aligned}\end{equation}
 Now, let us consider for $t \in [0,T)$
 \begin{equation}\label{eq:p_eps}
 p_\eps(t)= \frac{(T-t)^{\gamma-1}}{\Gamma(\gamma)}\nabla g(y_\eps(T))+\int_t^T \frac{(s-t)^{\gamma-1}}{\Gamma(\gamma)} a_\eps^\top (s)\underbrace{f{_\varepsilon}' (a_\eps y_\eps +\ell_\eps)(s) p_\eps(s)}_{=:\lambda_\eps(s)} \ds .
 \end{equation}
 To see that \eqref{eq:p_eps} admits a unique solution, we argue as in the proof of Proposition \ref{prop:existence}. From Lemma \ref{lem:bound} we know that the operator
$ z \mapsto \itt( a_\eps^\top f{_\varepsilon}' (a_\eps y_\eps +\ell_\eps) z) $ maps continuous functions to continuous functions,  since $a_\eps^\top , \ f{_\varepsilon}' (a_\eps y_\eps +\ell_\eps)  \in \li$.  However, the first term in \eqref{eq:p_eps} is only $L^r-$ integrable with $r$ given by Lemma \ref{lem:int}. This means that (no matter how smooth $z$ is) the fix point operator associated to \eqref{eq:p_eps}, namely
\[z \mapsto \frac{(T-\cdot)^{\gamma-1}}{\Gamma(\gamma)}\nabla g(y_\eps(T))+ \itt( a_\eps^\top f{_\varepsilon}' (a_\eps y_\eps +\ell_\eps) z) \]
 maps only to
 $L^r \ot$  with $r$ given by Lemma \ref{lem:int}. Due to Lemma \ref{lem:bound} we have for all $z_1,z_2 \in L^r(0,t^*;\mathbb{R}^n)$ the estimate
 \begin{equation*}
\begin{aligned}
\|\itt( a_\eps^\top f{_\varepsilon}' (a_\eps y_\eps +\ell_\eps)& (z_1-z_2))\|_{L^r(0,t^*;\mathbb{R}^n)} \\&\leq \frac{(t^*)^\gamma}{\gamma \Gamma (\gamma)} L\, \|a_\eps\|_{\li}\|z_1-z_2\|_{L^r(0,t^*;\mathbb{R}^n)}\end{aligned}\end{equation*}
and by arguing exactly  as in the proof of Proposition \ref{prop:existence} one obtains  that \eqref{eq:p_eps} admits a unique solution $p_\eps \in L^r \ot$ with $r$ given by Lemma \ref{lem:int}. This immediately implies that $\lambda_\eps \in L^r \ot,$ since $f{_\varepsilon}' (a_\eps y_\eps +\ell_\eps) \in \li$.

 Next, let ${\al} \in H^1(0,T;\mathbb{R}^{n \times n}) \times \KK$ be arbitrary but fixed. We abbreviate
 \[
 \delta y_\eps:=S_\eps' (a_\eps,\ell_\eps)(\al-\ale)
 \]
 and
 \[
 f{_\varepsilon}' (\cdot)(\cdot):=f{_\varepsilon}' (a_\eps y_\eps+\ell_\eps)( a _\eps \delta y_\eps+(a -a_\eps)y_\eps+\ell-\ell_\eps),
 \]
 which implies
 \begin{equation}\label{deltaye}
 \delta y_\eps =I_{0+}(f{_\varepsilon}' (\cdot)(\cdot)),
 \end{equation}
  on account of \eqref{eq:ode_lin_q_e}.
Now, in the light of Lemma \ref{fip} combined with the identities \eqref{eq:p_eps} and \eqref{deltaye} we have
\[
  \int_0^T  {( a_\eps^\top \lambda_\eps)(t)^\top} \delta y_\eps(t) \dt =\int_0^T  {f{_\varepsilon}' (\cdot)(\cdot)(t)^\top} \Big[p_\eps(t)-\frac{(T-t)^{\gamma-1}}{\Gamma(\gamma)}\nabla g(y_\eps(T))\Big]\dt.
\]
  Thus, we obtain
\begin{equation}\label{eqq1}
\begin{aligned}
\int_0^T {f{_\varepsilon}' (\cdot)(\cdot)(t)^\top}p_\eps(t) &-  {( a_\eps^\top \lambda_\eps)(t) ^\top}\delta y_\eps(t) \dt
\\&= {\nabla g(y_\eps(T))^\top} \underbrace{\int_0^T \frac{(T-t)^{\gamma-1}}{\Gamma(\gamma)} f ' (\cdot)(\cdot)(t) \dt }_{=\delta y_\eps(T), \text{ see }\eqref{deltaye}}.
\end{aligned}\end{equation}
Since $\lambda_\eps = f{_\varepsilon}' (a_\eps y_\eps +\ell_\eps) p_\eps$, we can simplify the left-hand side of the above equation as
\begin{equation}\label{eqq2}
\int_0^T {f{_\varepsilon}' (\cdot)(\cdot)(t)^\top}p_\eps(t) -  {( a_\eps^\top \lambda_\eps)(t) ^\top}\delta y_\eps(t) \dt
= \int_0^T  {\lambda_\eps(t)^\top}(  ( a -a_\eps)y_\eps+ \ell-\ell_\eps)(t) \dt.
\end{equation}
Inserting \eqref{eqq1} and \eqref{eqq2} in \eqref{necc} leads to
\begin{equation}\label{eq:sttt}
 \begin{aligned}
  0 & \le \int_0^T {f{_\varepsilon}' (\cdot)(\cdot)(t)^\top}p_\eps(t) -  {( a_\eps^\top \lambda_\eps)(t) ^\top}\delta y_\eps(t) \dt \\
    & \qquad +(2a_{\eps}-\bar a, a-a_\eps)_{H^1(0,T;\mathbb{R}^{n \times n})} +(2\ell_{\eps}- \ll, \ell-\ell_\eps)_{H^1(0,T;\mathbb{R}^{n})} \\
& =\dual{\lambda_\eps}{ ( a -a_\eps)y_\eps+\ell-\ell_\eps}_{L^r(0,T;\mathbb{R}^n)}
\\
& \qquad +(2a_{\eps}-\bar a,a-a_\eps)_{H^1(0,T;\mathbb{R}^{n \times n})} +(2\ell_{\eps}- \ll, \ell-\ell_\eps)_{H^1(0,T;\mathbb{R}^{n})}.
% \\\qquad \qquad  &\geq 0 \quad \forall\, \dal \in H^1(0,T;\mathbb{R}^{n \times n}) \times  \KK.
 \end{aligned}
 \end{equation}
% Note that in the first identity in \eqref{eq:sttt} we used the fact that the Jacobi matrix of $f_\eps$ is a diagonal matrix (which is due to the structure of $f_\eps$, see \eqref{feps}.)
% Now, let $\dal \in H^1(0,T;\mathbb{R}^{n \times n}) \times  \KK$ be arbitrary but fixed.
Setting $(a,\ell)$ in \eqref{eq:sttt} to $(a_\eps \pm \da,\ell_\eps)$, $\da \in H^1(0,T;\mathbb{R}^{n \times n})$, and $(a_\eps, \ell)$, $\ell\in\KK$, respectively, yields
\begin{subequations}\label{eq:epsss}
 \begin{gather}
(2 a_\eps-\bar a,\delta a)_{H^1(0,T;\mathbb{R}^{n \times n})}+ \dual{ \lambda_\eps}{ \delta a \,y_\eps}_{L^r(0,T;\mathbb{R}^{n })}=0 \quad \forall\,\delta a \in H^1(0,T;\mathbb{R}^{n \times n}),
\\[1mm] (2\ell_{\eps}- \ll, \ell-\ell_\eps)_{H^1(0,T;\mathbb{R}^n)}+ \dual{\lambda_\eps}{ \ell -\ell_\eps}_{L^r(0,T;\mathbb{R}^{n })} \geq 0 \quad \forall\,  \ell \in \KK.
 \end{gather}
 \end{subequations}
The next step is to show that $p_\eps$ and hence $\lambda_\eps$ are bounded independently of $\eps$.
From \eqref{eq:p_eps} we further obtain
% \begin{equation}\begin{aligned}
%  p_\eps(t) &=  \frac{(T-t)^{\gamma-1}}{\Gamma(\gamma)}\nabla g(y_\eps(T))
%  \\&\qquad +\int_{0}^{T-t} \frac{((T-t)-s)^{\gamma-1}}{\Gamma(\gamma)} a_\eps^T(T-s)f{_\varepsilon}' (a_\eps \bar y_\eps +\ell_\eps)(T-s) p_\eps(T-s) \ds,
%\\&=  \frac{t^{\gamma-1}}{\Gamma(\gamma)}(y_\eps(T-t)-y_d)+\int_{T-t}^{T} \frac{(s-(T-t))^{\gamma-1}}{\Gamma(\gamma)} a_\eps^T(s)\underbrace{f{_\varepsilon}' (a_\eps \bar y_\eps +\ell_\eps)(s) p_\eps(s)}_{=:\lambda_\eps} \ds \quad \forall\, t \in [0,T).
%\end{aligned} \end{equation}which means that
 \begin{equation*}\begin{aligned}
  p_\eps(T-t) &=  \frac{t^{\gamma-1}}{\Gamma(\gamma)}\nabla g(y_\eps(T))
  \\&+\int_{0}^{t} \frac{(t-s)^{\gamma-1}}{\Gamma(\gamma)} a_\eps^\top (T-s)f{_\varepsilon}' (a_\eps \bar y_\eps +\ell_\eps)(T-s) p_\eps(T-s) \ds \quad \forall\,t \in (0,T].\end{aligned} \end{equation*}
We abbreviate $\widetilde p_\eps:=p(T-\cdot)$. Since $\|f{_\varepsilon}' (a_\eps y_\eps+\ell_\eps)\|_{\li} \leq L$, we have the estimate  \begin{equation*}\begin{aligned}
 \| \widetilde p_\eps(t) \| \leq c_1 t^{\gamma-1}+\int_{0}^{t} \frac{(t-s)^{\gamma-1}}{\Gamma(\gamma)} c_2\,L\,\|\widetilde p_\eps(s)\| \ds,\quad t\in (0,T],\end{aligned} \end{equation*}
 where $c_1,c_2>0$ are independent of $\eps$. Here we also used that $\{\|a_\eps\|_{\li}\}$, $\{ \|\nabla g(y_\eps(T))\|_{\mathbb{R}^n}\}$ are uniformly bounded with respect to $\eps$, cf., \eqref{al_conv} and \eqref{y_conv}; recall that $g:\mathbb{R}^n \to \mathbb{R}$ is continuously differentiable, by assumption.
In view  of Lemma \ref{lem:g}, this implies
 \begin{equation*}\begin{aligned}
 \| \widetilde p_\eps(t) \| \leq C\,c_1  t^{\gamma-1},\quad t\in (0,T].\end{aligned} \end{equation*}
 Thus, by employing again $\|f{_\varepsilon}' (a_\eps y_\eps+\ell_\eps)\|_{\li} \leq 1,$ we have  $\|\lambda_\eps\|_{L^r \ot} \leq \tilde C$ with $r\in (1,\frac1{1-\gamma})$ given by Lemma \ref{lem:int}, and we can extract a weakly  convergent subsequence
 \begin{equation*}\begin{aligned}
 \lambda_\eps \weakly^* \lambda \quad \text{in }L^r \ot. \end{aligned} \end{equation*}
 Passing to the limit in \eqref{eq:epsss} and using \eqref{al_conv}, \eqref{y_conv} now yields
 \begin{align*}
(\bar a,\delta a)_{H^1(0,T;\mathbb{R}^{n \times n})}+  {\dual{ \lambda}{\delta a \,\bar y}_{L^r(0,T;\mathbb{R}^{n \times n})}}&=0 \quad \forall\,\delta a \in {H^1(0,T;\mathbb{R}^{n \times n})},
\\[1mm] (\ll, \ell-\bar \ell)_{H^1(0,T;\mathbb{R}^n)}+  {\dual{\lambda}{  \ell -\bar \ell}_{L^r(0,T;\mathbb{R}^{n })}} &\geq 0 \quad \forall\,  \ell \in \KK.
 \end{align*}
 The proof is now complete.
\end{proof}

% \end{appendix}
\bibliographystyle{plain}
\bibliography{strong_stat_coupled_pde}

\begin{thebibliography}{10}

\bibitem{OPAgrawal_2004a}
O.~P. Agrawal.
\newblock A general formulation and solution scheme for fractional optimal
  control problems.
\newblock {\em Nonlinear Dynam.}, 38(1-4):323--337, 2004.

\bibitem{fb}
H.~Antil, T.~S. Brown, R.~Lohner, F.~Togashi, and D.~Verma.
\newblock Deep neural nets with fixed bias configuration.
\newblock {\em Numerical Algebra, Control and Optimization}, 2022.

\bibitem{HAntil_HDiaz_Eherberg_2022a}
H.~Antil, H.~D{\'\i}az, and E.~Herberg.
\newblock An {Optimal} {Time} {Variable} {Learning} {Framework} for {Deep}
  {Neural} {Networks}.
\newblock Technical report, arXiv, April 2022.
\newblock arXiv:2204.08528.

\bibitem{HAntil_HCElman_AOnwunta_DVerma_2021a}
H.~Antil, H.~C. Elman, A.~Onwunta, and D.~Verma.
\newblock Novel deep neural networks for solving bayesian statistical inverse
  problems.
\newblock Technical report, arXiv, February 2021.
\newblock arXiv:2102.03974.

\bibitem{unified}
H.~Antil, C.~G. Gal, and M.~Warma.
\newblock A unified framework for optimal control of fractional in time
  subdiffusive semilinear {PDE}s.
\newblock {\em Discrete and Continuous Dynamical Systems Series S},
  15(8):1883--1918, 2022.

\bibitem{HAntil_RKhatri_RLohner_DVerma_2020a}
H.~Antil, R.~Khatri, R.~L{\"o}hner, and D.~Verma.
\newblock Fractional deep neural network via constrained optimization.
\newblock {\em Machine Learning: Science and Technology}, 2(1):015003, dec
  2020.

\bibitem{HAntil_EOtarola_AJSalgado_2015a}
H.~Antil, E.~Ot{\'a}rola, and A.~J. Salgado.
\newblock A {S}pace-{T}ime {F}ractional {O}ptimal {C}ontrol {P}roblem:
  {A}nalysis and {D}iscretization.
\newblock {\em SIAM J. Control Optim.}, 54(3):1295--1328, 2016.

\bibitem{Barbu:1981:NCD}
V.~Barbu.
\newblock Necessary conditions for distributed control problems governed by
  parabolic variational inequalities.
\newblock {\em SIAM Journal on Control and Optimization}, 19(1):64--86, 1981.

\bibitem{barbu}
V.~Barbu.
\newblock {\em Optimal control of variational inequalities}.
\newblock Research notes in mathematics 100, Pitman, Boston-London-Melbourne,
  1984.

\bibitem{st_coup}
L.~Betz.
\newblock Strong stationarity for optimal control of a non-smooth coupled
  system: Application to a viscous evolutionary {VI} coupled with an elliptic
  {PDE}.
\newblock {\em SIAM J.\ on Optimization}, 29(4):3069--3099, 2019.

\bibitem{mcrf}
L.~Betz.
\newblock Strong stationarity for a highly nonsmooth optimization problem with
  control constraints.
\newblock {\em Mathematical Control and Related Fields,
  doi={10.3934/mcrf.2022047}}, 2022.

\bibitem{by18}
L.~Betz and I.~Yousept.
\newblock Optimal control of elliptic variational inequalities with bounded and
  unbounded operators.
\newblock {\em Mathematical Control and Related Fields}, 11(3):479--498, 2021.

\bibitem{Bittner1975}
L.~Bittner.
\newblock On optimal control of processes governed by abstract functional,
  integral and hyperbolic differential equations.
\newblock {\em Math. Operationsforsch. Statist.}, 6(1):107--134, 1975.

\bibitem{cc}
C.~Christof.
\newblock Sensitivity analysis and optimal control of obstacle-type evolution
  variational inequalities.
\newblock {\em SIAM J. Control Optim.}, 57(1):192--218, 2019.

\bibitem{brok_ch}
C.~Christof and M.~Brokate.
\newblock Strong stationarity conditions for optimal control problems governed
  by a rate-independent evolution variational inequality.
\newblock arXiv:2205.01196, 2022.

\bibitem{cr_meyer}
C.~Christof, C.~Clason, C.~Meyer, and S.~Walther.
\newblock Optimal control of a non-smooth, semilinear elliptic equation.
\newblock {\em Mathematical Control and Related Fields}, 8(1):247--276, 2018.

\bibitem{quasi_nonsmooth}
C.~Clason, V.H. Nhu, and A.~R\"osch.
\newblock Optimal control of a non-smooth quasilinear elliptic equation.
\newblock {\em Mathematical Control and Related Fields}, 11(3):521--554, 2021.

\bibitem{DelosReyes-Meyer}
J.~C. De~los Reyes and C.~Meyer.
\newblock Strong stationarity conditions for a class of optimization problems
  governed by variational inequalities of the second kind.
\newblock {\em Journal of Optimization Theory and Applications},
  168(2):375--409, 2016.

\bibitem{diethelm}
K.~Diethelm.
\newblock {\em The Analysis of Fractional Differential Equations}.
\newblock Lecture Notes in Mathematics. Springer, Berlin, 2004.

\bibitem{gronwall}
C.~M. Elliott and S.~Larsson.
\newblock Error estimates with smooth and nonsmooth data for a finite element
  method for the {C}ahn-{H}illiard equation.
\newblock {\em Mathematics of Computation}, 58(198):603--630, 1992.

\bibitem{evans}
L.C. Evans.
\newblock {\em Partial Differential Equations}.
\newblock American Mathematical Society, 2010.

\bibitem{gk}
C.~Geiger and C.~Kanzow.
\newblock {\em Theorie und Numerik restringierter Optimierungsaufgaben}.
\newblock Springer, 2002.

\bibitem{GoldbergTroltzsch1989}
H.~Goldberg and F.~Tr\"{o}ltzsch.
\newblock Second order optimality conditions for a class of control problems
  governed by nonlinear integral equations with application to parabolic
  boundary control.
\newblock {\em Optimization}, 20(5):687--698, 1989.

\bibitem{He_et_al_2016a}
K.~He, X.~Zhang, S.~Ren, and J.~Sun.
\newblock Deep residual learning for image recognition.
\newblock In {\em 2016 IEEE Conference on Computer Vision and Pattern
  Recognition (CVPR)}, pages 770--778, 2016.

\bibitem{henry}
D.~Henry.
\newblock {\em Geometric Theory of Semilinear Parabolic Equations}.
\newblock Springer Verlag, Berlin, 1981.

\bibitem{hmw13}
R.~Herzog, C.~Meyer, and G.~Wachsmuth.
\newblock B- and strong stationarity for optimal control of static plasticity
  with hardening.
\newblock {\em SIAM J. Optim.}, 23(1):321--352, 2013.

\bibitem{HintermuellerKopacka2008:1}
M.~Hinterm{\"u}ller and I.~Kopacka.
\newblock Mathematical programs with complementarity constraints in function
  space: {C}- and strong stationarity and a path-following algorithm.
\newblock {\em SIAM Journal on Optimization}, 20(2):868--902, 2009.

\bibitem{ItoKunisch2008:1}
K.~Ito and K.~Kunisch.
\newblock Optimal control of parabolic variational inequalities.
\newblock {\em Journal de Math{\'e}matiques Pures et Appliqu{\'e}s},
  93(4):329--360, 2010.

\bibitem{book}
A.~A. Kilbas, H.~M. Srivastava, and J.~J. Trujillo.
\newblock {\em Theory and applications of fractional differential equations},
  volume 204 of {\em North-Holland Mathematics Studies}.
\newblock Elsevier Science B.V., Amsterdam, 2006.

\bibitem{paper}
C.~Meyer and L.M. Susu.
\newblock Optimal control of nonsmooth, semilinear parabolic equations.
\newblock {\em SIAM Journal on Control and Optimization}, 55(4):2206--2234,
  2017.

\bibitem{mp76}
F~Mignot.
\newblock Contr{\^{o}}le dans les in{\'{e}}quations variationelles elliptiques.
\newblock {\em Journal of Functional Analysis}, 22(2):130--185, 1976.

\bibitem{mp84}
F.~Mignot and J.-P. Puel.
\newblock Optimal control in some variational inequalities.
\newblock {\em SIAM Journal on Control and Optimization}, 22(3):466--476, 1984.

\bibitem{ruthotto2020deep}
L.~Ruthotto and E.~Haber.
\newblock Deep neural networks motivated by partial differential equations.
\newblock {\em Journal of Mathematical Imaging and Vision}, 62(3):352--364,
  2020.

\bibitem{ScheelScholtes2000}
H.~Scheel and S.~Scholtes.
\newblock Mathematical programs with complementarity constraints: Stationarity,
  optimality, and sensitivity.
\newblock {\em Mathematics of Operations Research}, 25(1):1--22, 2000.

\bibitem{schirotzek}
W.~Schirotzek.
\newblock {\em Nonsmooth Analysis}.
\newblock Springer, Berlin, 2007.

\bibitem{stynes}
M.~Stynes.
\newblock Too much regularity may force too much uniqueness.
\newblock {\em Fractional Calculus and Applied Analysis}, 19(6):1554--1562,
  2016.

\bibitem{tiba}
D.~Tiba.
\newblock {\em Optimal control of nonsmooth distributed parameter systems}.
\newblock Springer, 1990.

\bibitem{troe}
F.~Tr{\"o}ltzsch.
\newblock {\em Optimal Control of Partial Differential Equations}, volume 112
  of {\em Graduate Studies in Mathematics}.
\newblock American Mathematical Society, Providence, 2010.
\newblock Theory, methods and applications, Translated from the 2005 German
  original by J{\"u}rgen Sprekels.

\bibitem{wachsm_2014}
G.~Wachsmuth.
\newblock Strong stationarity for optimal control of the obstacle problem with
  control constraints.
\newblock {\em SIAM Journal on Optimization}, 24(4):1914--1932, 2014.

\bibitem{e_qvi}
G.~Wachsmuth.
\newblock Elliptic quasi-variational inequalities under a smallness assumption:
  uniqueness, differential stability and optimal control.
\newblock {\em Calc. Var. Partial Differential Equations}, 59(2):Paper No. 82,
  15, 2020.

\bibitem{Wolfersdorf1976}
L.~Wolfersdorf.
\newblock Optimal control of a class of processes described by general integral
  equations of {H}ammerstein type.
\newblock {\em Math. Nachr.}, 71:115--141, 1976.

\bibitem{YeGaoDing2007}
H.~Ye, J.~Gao, and Y.~Ding.
\newblock A generalized {G}ronwall inequality and its application to a
  fractional differential equation.
\newblock {\em J. Math. Anal. Appl.}, 328(2):1075--1081, 2007.

\end{thebibliography}

\end{document}